\documentclass[12pt,epsfig,amsfonts]{amsart}
\usepackage[utf8]{inputenc}
\usepackage{amsmath,amsthm,amssymb,
mathrsfs,amscd,epsfig,color,bbm}
\usepackage{url}
\usepackage{ulem}
\usepackage{graphicx}

\newcommand{\cP}{\mathcal{P}}

\newtheorem{prop}{Proposition}[section]
\newtheorem*{maintheorem}{Main Theorem}
\newtheorem{thm}[prop]{Theorem}

\newtheorem{lemma}[prop]{Lemma}

\newtheorem{sublemma}[prop]{Sublemma}

\theoremstyle{remark}

\numberwithin{equation}{section}

\title[Polynomial rate of mixing for the heterochaos baker maps]{Polynomial rate of mixing for the heterochaos baker maps with 
mostly neutral center}
\author{Hiroki Takahasi and Masato Tsujii}
\date{\today}

\address{Department of Mathematics,
Keio University, Yokohama,
223-8522, JAPAN} 
\email{hiroki@math.keio.ac.jp}
\address{Department of Mathematics, Kyushu University, Fukuoka, 819-0395, JAPAN} 
\email{tsujii@math.kyushu-u.ac.jp}


\begin{document}
\begin{abstract}
For the heterochaos baker maps  whose central direction is mostly neutral, we prove that correlations for H\"older continuous functions decay at an optimal polynomial rate of order $n^{-3/2}$. 
Our method of proof relies on a description of the action of a reduced Perron-Frobenius operator by means of a comparison to the symmetric simple random walk with an absorbing wall, aka `gambler's ruin problem'.
\end{abstract}
\maketitle

\section{Introduction}
The heterochaos baker maps \cite{STY21,S60,TY23} are piecewise affine maps on the square $[0,1]^2$ or the cube $[0,1]^3$. 
Like Smale's horseshoe map \cite{Sma98} which serves as a model for uniformly hyperbolic systems, 
the heterochaos baker maps are one of the simplest models for partially hyperbolic systems. Despite the complexity of their dynamics, the simplicity of their definitions enables a rather precise description of their dynamical properties.

Decay of correlations is a key metric for gauging the level of chaos. Fast decay rates indicate that the associated time series become asymptotically independent. Exponential decay of correlations is usually obtained as a consequence of some hyperbolicity of the underlying system \cite{You98}. 
Typical examples of systems
with subexponential decay of correlations are maps with neutral fixed points (see e.g., \cite{Gou04,LSV99,PomMan80,Sar02,You99}), and the mechanism that slows down decay rates is essentially one-dimensional.
To shed lights on higher dimensional mechanisms leading to subexponential decay of correlations,
new concrete examples 
are required.

For the heterochaos baker maps
with mostly expanding or contracting center,
exponential decay of correlations was established in \cite{T23}. 
The aim of this paper is to establish subexponential decay of correlations for the maps with mostly neutral center. 
Below we introduce the heterochaos baker maps, and state our main result. 

\subsection{The heterochaos baker maps}\label{hetero-def}
Let $M\geq2$ be an integer. We define the hetrochaos baker maps $f_a\colon[0,1]^2\to[0,1]^2$ and $f_{a,b}\colon[0,1]^3\to[0,1]^3$, where both parameters $a$, $b$ range over the interval $(0,\frac{1}{M})$.
We write $(x_u,x_c)$ and $(x_u, x_c, x_s)$ for the coordinates on $[0,1]^2$ and $[0,1]^3$ respectively.

For $a\in (0,\frac{1}{M})$, define a piecewise affine expanding map $\tau_a\colon[0,1]\to[0,1]$ by
\[
\tau_a(x_u)=\begin{cases}\vspace{1mm}
\displaystyle{\frac{x_u-(k-1)a}{a}}&\text{ on }[(k-1)a,ka),\ k\in\{1,\ldots,M\},\\ \displaystyle{\frac{x_u-Ma}{1-Ma}}&\text{ on }
[Ma,1].
\end{cases}
\]
 Consider a set $D$ of $2M$ symbols
\[
D=\{\alpha_1,\ldots,\alpha_M\}\cup\{\beta_1,\ldots,\beta_M\}.
\]
For each symbol $\gamma\in D$, define a rectangular domain $\Omega_\gamma^+$ in $[0,1]^2$ by
\[\Omega_{\alpha_k}^+=\left[(k-1)a,ka\right)\times
\left[0,1\right]\text{ for }k\in\{1,\ldots,M\}
\]
and
\[\Omega_{\beta_k}^+=\begin{cases}\vspace{1mm}
\left[Ma,1\right]\times
\displaystyle{\left[\frac{k-1}{M },\frac{k }{M }\right)}&
\text{ for }k\in\{1,\ldots,M-1\},\\
\left[Ma,1\right]\times
\displaystyle{\left[\frac{M-1 }{M},1\right]}&\text{ for }k=M.
\end{cases}\]
We now define $f_a\colon[0,1]^2\to [0,1]^2$ by
\[\begin{split}
  f_a(x_u,x_c)=
  \begin{cases}
    \displaystyle{\left(\tau_a(x_u),\frac{x_c}{M}+\frac{k-1}{M}\right)}&\text{ on }\Omega_{\alpha_k}^+,\ k\in\{1,\ldots,M\},\\
   \displaystyle{\left (\tau_a(x_u),Mx_c-k+1\right)}&\text{ on }\Omega_{\beta_k}^+,\ k\in\{1,\ldots,M\}.
   \end{cases}
\end{split}\]
The map $f_a$ in the case $M=2$ is illustrated in \textsc{Figure}~\ref{hetero2}.

The maps $f_{a,b}$ are extensions of $f_{a}$ as skew products with contracting affine maps.
Put $\Omega_\gamma=\Omega_\gamma^+\times
[0,1]$ for $\gamma\in D$. We define $
f_{a,b}\colon[0,1]^3\to [0,1]^3$
 by
\[\begin{split}
  f_{a,b}(x_u,x_c,x_s)=
  \begin{cases}
    (f_{a}(x_u,x_c),
    (1-Mb)x_s)&\text{ on }\Omega_{\alpha_k},\ k\in\{1,\ldots,M\}
    ,\\
   \displaystyle{\left (f_{a}(x_u,x_c),bx_s+1+b(k-M-1)\right)}&\text{ on }\Omega_{\beta_k},\ k\in\{1,\ldots,M\}.
   \end{cases}
\end{split}\]
The map $f_{a,b}$ in the case $M=2$ is illustrated in \textsc{Figure}~\ref{hetero3}. 
\begin{figure}
\begin{center}
\includegraphics[height=4cm,width=12cm]
{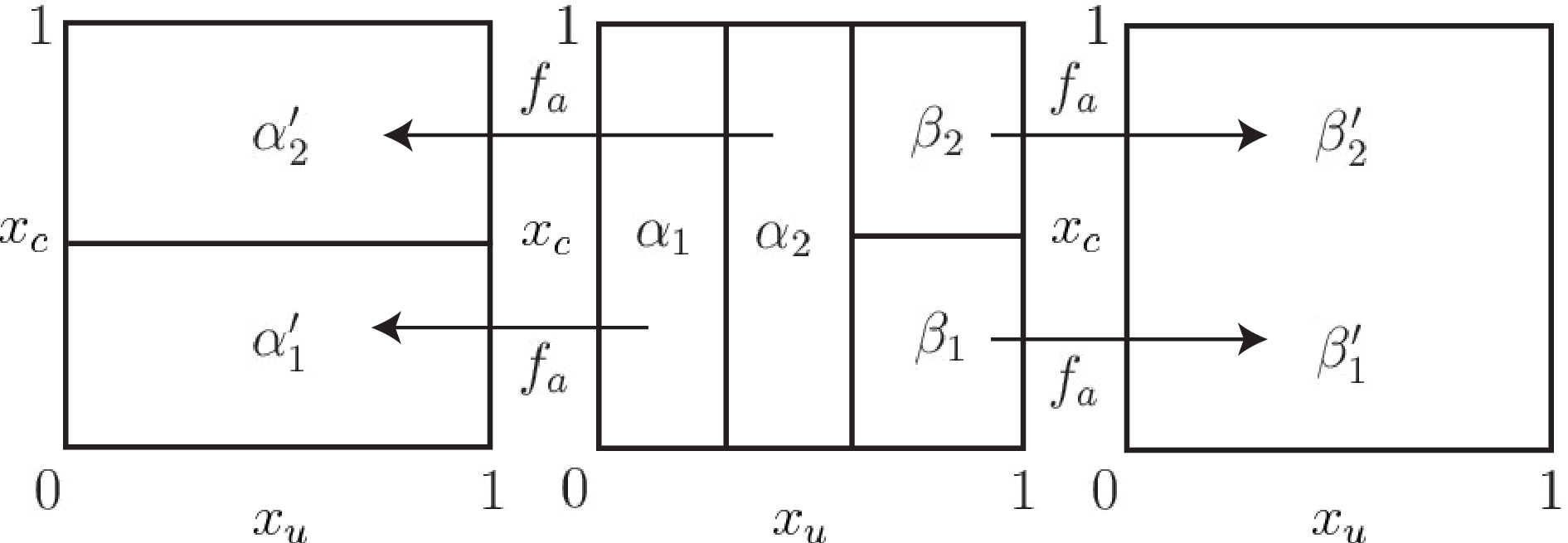}
\caption
{The heterochaos baker map $f_{a}$ with $M=2$. For each $\gamma\in D$,
the domain $\Omega_\gamma^+$  and its image are labeled with $\gamma$ and $\gamma'$ respectively.}\label{hetero2}
\end{center}
\end{figure}
 \begin{figure}
\begin{center}
\includegraphics[height=5cm,width=13cm]
{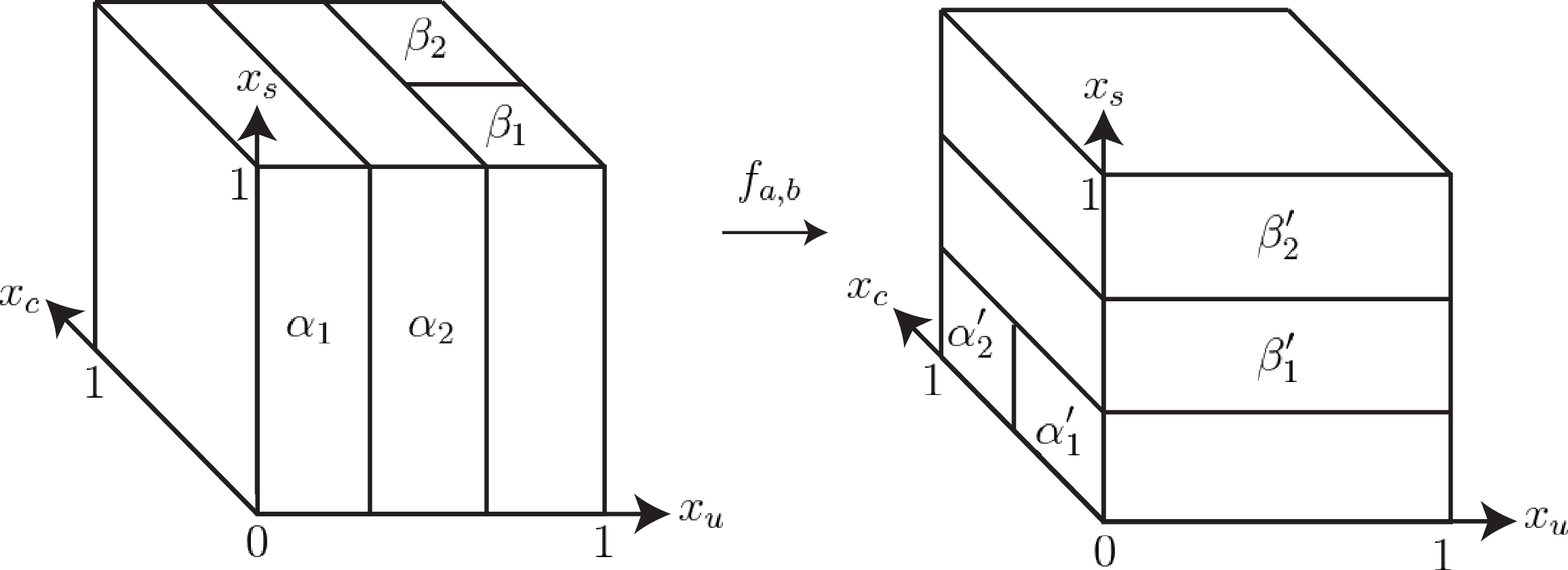}
\caption
{The heterochaos baker map $f_{a,b}$ with $M=2$. For each $\gamma\in D$,
the domain $\Omega_\gamma$  and its image are labeled with $\gamma$ and $\gamma'$ respectively. }\label{hetero3}
\end{center}
\end{figure}

The maps $f_a$ and $f_{a,b}$ with $M=2$, $a=\frac{1}{3}$, $b=\frac{1}{6}$ 
were introduced in \cite{STY21}, and the above generalized definitions were introduced in \cite{TY23}.  In \cite{STY22}, some variants of $f_{a,b}$ were considered in the context of fractal geometry and homoclinic bifurcations of three dimensional diffeomorphisms. Iterated function systems intimately related to $f_{a,b}$ was considered in \cite{HK22}.
  In  \cite{DG12},
another model of partially hyperbolic systems was introduced and investigated. 
For a comprehensive review on the dynamics of the heterochaos baker maps and more references, see \cite{S60}.

 The dynamics of $f=f_{a,b}$ is partially hyperbolic in the following sense.
Under the iteration of $f$,
the $x_u$-direction is expanding and the $x_s$-direction is contracting, while 
the $x_c$-direction is central:
contracting by factor $\frac{1}{M}$ on $\bigcup_{k=1}^M\Omega_{\alpha_k}$ and expanding by factor $M$ on $\bigcup_{k=1}^M\Omega_{\beta_k}$.
Since $f$ is a skew product over the one-dimensional map $\tau_a$, 
the ergodicity of $\tau_a$ with respect to the Lebesgue measure on $[0,1]$ implies
 \[
 \lim_{n\to\infty}\frac{1}{n}\#\left\{0\le i<n\colon f^i(x)\in\bigcup_{k=1}^M\Omega_{\alpha_k}\right\}=aM 
  \text{ for Lebesgue a.e. $x\in[0,1]^3$.}
  \]  If $a\in(0,\frac{1}{2M})$ (resp. $a\in(\frac{1}{2M},\frac{1}{M})$), then the $x_c$-direction is non-uniformly expanding (resp. contracting) for Lebesgue almost every initial point.
  We classify $f$ into three types according to the dynamics in the $x_c$-direction:
\begin{itemize}
\item $a\in(0,\frac{1}{2M})$ (mostly expanding center),

 \item $a\in(\frac{1}{2M},\frac{1}{M})$
 (mostly contracting center),
 \item $a=\frac{1}{2M}$ (mostly neutral center).
 \end{itemize}
 The parameter $a=\frac{1}{2M}$ is a bifurcation point at which the `central Lyapunov exponent' crosses zero.
 Bifurcations of similar types that involve a Lyapunov exponent crossing zero were considered in physics literature \cite{AAN98,OS94,Pik84,YF84}, 
 and recently in \cite{HK22,TZ} with mathematically rigorous treatments.


\subsection{Statement of the result}
In what follows we assume that $f_{a,b}$ preserves the Lebesgue measure on $[0,1]^3$, denoted by $m$. This assumption is equivalent to the condition $a+b=\frac{1}{M}$ (See \cite[Lemma~3.14]{S60}). 
For a pair $(\varphi,\psi)$ of functions in $L^2([0,1]^3)$, 
define their correlation 
by 
\[
\mathrm{Cor}(\varphi,\psi)=\left|\int \varphi \psi {\rm d}m -\int \varphi {\rm d}m  \int \psi {\rm d}m\right|.
\]
Note that $\mathrm{Cor}(\varphi,\psi)=0$ if and only if $\varphi$ and $\psi$ are independent with respect to $m$.
Since $(f,m)$ is mixing \cite{S60,T23},
for any pair $(\varphi,\psi)$ of  $L^2$ functions we have
$\mathrm{Cor}(\varphi,\psi\circ f^n)\to0$ as $n\to\infty$.
Of interest to us is how fast
$\mathrm{Cor}(\varphi,\psi\circ f^n)$ decays as $n\to\infty$.
For a general pair $(\varphi,\psi)$ of $L^2$ functions, the decay of their correlation can become arbitrarily slow. To obtain a reasonable result, we 
assume that $\varphi$ and $\psi$ are H\"older continuous. 
For $\theta\in(0,1]$ and $d=1,2,3$ let $C^\theta([0,1]^d)$ denote the space of H\"older continuous functions on $[0,1]^d$ with exponent $\theta$ equipped with the H\"older norm $\|\varphi\|_{C^\theta}=\|\varphi\|_{L^\infty} + |\varphi|_{C^\theta}$, where
\[
{\|\varphi\|_{L^\infty}=\sup_{x\in[0,1]^d}|\varphi(x)|} \ \text{ and }\ 
|\varphi|_{C^\theta}=
\sup_{\substack{x,y\in[0,1]^d\\x\neq y }} \frac{|\varphi(x)-\varphi(y)|}{|x-y|^\theta},
\]
and
 $|x-y|$ denotes the Euclidean distance between $x$ and $y$ in $[0,1]^d$.

  Based on a numerical experiment, it was conjectured in \cite{S60} that the decay rate of  correlations 
  for H\"older continuous functions
is of polynomial order $n^{-3/2}$, and that this rate is optimal.
 The following theorem settles this conjecture. 

\begin{maintheorem}\label{thm-1}
For any $\theta\in (0,1]$ there exists $C>0$ such that for any pair $(\varphi,\psi)$ of functions in $C^\theta([0,1]^3)$ and
the map $f=f_{\frac{1}{2M},\frac{1}{2M}}$ with mostly neutral center, 
we have 
\begin{equation}\label{eq:upperbound}
\mathrm{Cor}(\varphi,\psi\circ f^n)\le C  \|\varphi\|_{C^\theta} \|\psi\|_{C^\theta} n^{-3/2 }\quad \text{for all $n\ge 1$}.
\end{equation}
The exponent $-3/2$ is optimal, in that
there exist a pair $(\varphi,\psi)$ of functions in $C^\theta([0,1]^3)$ and a constant $C(\varphi,\psi)>0$ such that we have 
\begin{equation}\label{eq:lowerbound}
\mathrm{Cor}(\varphi,\psi\circ f^n) \ge C(\varphi,\psi) n^{-3/2}\quad \text{for all $n\ge 1$.}
\end{equation}
\end{maintheorem}

The above Main~Theorem complements \cite[Theorem~B]{T23} that established
the exponential decay of correlations for H\"older continuous functions for maps with mostly expanding or contracting center. 

\begin{thm}[\cite{T23}, Theorem~B]\label{thm-exp-mix}
If $a\in(0,\frac{1}{M})\setminus\{\frac{1}{2M}\}$, 
then for any $\theta\in(0,1]$ there exists $\lambda\in(0,1)$ such that
 any pair $(\varphi,\psi)$ of functions in $C^\theta([0,1]^3)$, there exists $C=C(\varphi,\psi)>0$ such that for $f=f_{a,\frac{1}{M}-a}$ we have
\[
{\rm Cor}(\varphi,\psi\circ f^n)\leq C\lambda^n\quad \text{ for all $n\geq1$.}
\]
\end{thm}

In \cite{T23}, Theorem~\ref{thm-exp-mix} was proved by constructing a tower with exponential tails, and then applying the general result of Young \cite{You98}. 
For known results on exponential decay of correlations for some partially hyperbolic diffeomorphisms obtained by applying Young's result \cite{You98}, see \cite{C02,C04,Do00} for example. 
In the case $a=\frac{1}{2M}$ we can still construct a tower, 
but the Lebesgue measure cannot be lifted to a finite measure on the tower,
 and so
Young's general result \cite{You99} on subexponential decay of correlations does not apply. 


\subsection{Outline of the proof of the Main Theorem and organization of the paper}
The rest of this paper consists of three sections.  
We only consider the heterochaos baker map on $[0,1]^3$ with mostly neutral center
\[
f=f_{\frac{1}{2M},\frac{1}{2M}}. 
\]
Throughout Sections~$2$ and $3$ we assume $M=2$.
In Section~$2$ we begin by introducing a reduced Perron-Frobenius operator 
that extracts  
the dynamics of $f$ in the $x_c$-direction. 
Interestingly, the action of this operator can be described by the Markov process of the symmetric simple random walk with an absorbing wall, aka `gambler's ruin problem' (see \cite{Fel}). With this description we deduce key polynomial estimates on the action of the reduced operator. In Section~$3$, from these estimates we recover polynomial estimates on the action of the original Perron-Frobenius operator, and prove the Main Theorem in the case $M=2$.
In Section~$4$ we indicate necessary minor modifications to deal with the case $M>2$. 

\subsection{Remarks on $T$, $T^{-1}$ transformations} The heterochaos baker maps are closely related to $T$, $T^{-1}$ transformations \cite{Kal82,Rud88}. 
Let $(X,S,\mu)$ and $(Y,T,\nu)$ be two ergodic probability preserving transformations, with $T$ invertible and let $h\colon X\to\mathbb Z$ be measurable.
The associated $T,T^{-1}$ transformation is the map \[(x,y)\in X\times Y\mapsto (S(x),T^{h(x)}y)\in X\times Y.\] For instance, one may take $X=Y=[0,1]$, $S = \tau$,  $h =
\mathbbm{1}_{[\frac{1}{2},1]} - \mathbbm{1}_{[0,\frac{1}{2})}$, where $\mathbbm{1}_{A}$ denotes the indicator function for a set $A\subset [0,1]$.
The definition does not work with $T(y) = My$ mod $1$ as $T$ is not invertible, but one may be able to subvert this obstacle by using the first coordinate $x_u$ to decide which branch of $T^{-1}$ to use. With this kind of analogy, the case with non-zero drift
($\int h{\rm d}\mu\neq0$) corresponds to the case $a\neq\frac{1}{2M}$ and the case with zero drift $(\int h{\rm d}\mu=0$)
corresponds to the case $a=\frac{1}{2M}$.
The mechanism for exponential mixing of $f_{a}$, $f_{a,\frac{1}{M}-a}$ with $a\neq\frac{1}{2M}$
proved in \cite{T23} is similar to that of generalized $T,T^{-1}$ transformations with non-zero drift \cite[Theorem~4.1(a)]{DDKN22}. For generalized $T$, $T^{-1}$ transformations
with zero drift, polynomial mixing rates faster than $n^{-1}$ were obtained in
\cite[Theorem~2.1]{L-Bor06}, \cite[Theorem~4.7]{DDKN22}.

\section{Polynomial estimates on Reduced Perron-Frobenius operator}
\label{sec2}
The aim of this section is to establish
 polynomial decay estimates on the reduced Perron-Frobenius operator.\smallskip

\noindent{\bf Hypothesis throughout Section~2 and Section~3:  $M=2$.}\smallskip
 
\noindent In Section~\ref{reduced-sec} we introduce this operator $\cP_0$, together with the original Perron-Frobenius operator $\cP$, and state the main polynomial decay estimates in Theorem~\ref{Th2}. In Section~\ref{haar-sec} we introduce a decomposition of the function space on which $\cP_0$ acts, and in Section~\ref{rPF-sec} start to analyze this action. In Section~\ref{gamble} we compare the action of $\cP_0$ with the symmetric simple random walk with an absorbing wall, and deduce key estimates.
 We establish a polynomial lower bound in Section~\ref{low-sec}, and
 complete the proof of the statement of Theorem~\ref{Th2} in Section~\ref{pfthm-sec}.
\medskip

\subsection{Reduced Perron-Frobenius operator}\label{reduced-sec}
Let  $f\colon[0,1]^3\to [0,1]^3$ be the heterochaos baker map with mostly neutral center. The Perron-Frobenius operator 
$\cP\colon L^1([0,1]^3)\to L^1([0,1]^3)$ associated to $(f,m)$ is given by
\[\cP u=\frac{{\rm d}}{{\rm d}m}\int_{f^{-1}(\cdot)}u{\rm d}m
\ \text{ for }u\in L^1([0,1]^3).\]
Disregarding sets with Lebesgue measure zero, we have
\begin{equation}\label{PF}
\cP u=u\circ f^{-1}.
\end{equation}
For  $u,v\in L^2([0,1]^3)$ we have
\[
\langle u,v\circ f\rangle=\langle \cP u,v\rangle,
\]
where $\langle \cdot,\cdot\rangle$ denotes the $L^2$ inner product:  
$\langle u,v\rangle=\int u\cdot \overline{v}\; {\rm d}m.$

Let $L^1_c([0,1]^3)$ denote the subspace of $L^1([0,1]^3)$ that consists of functions depending only on $x_c$. We naturally  
identify $L^1_c([0,1]^3)$ with $L^1([0,1])$. 
Let $\Pi\colon L^1([0,1]^3)\to L^1_c([0,1]^3)$ denote the projection defined by 
\[
(\Pi u)(x_c)=\int_{[0,1]} \int_{[0,1]}  u(s,x_c,t) {\rm d}s {\rm d}t.
\]
We introduce a {\it reduced Perron-Frobenius operator} 
\[
\cP_0=\Pi\circ \cP:L^1_c([0,1]^3)\to L^1_c([0,1]^3).
\]
We will give a precise description of this operator.
 Let 
\[
L^2_0([0,1])=\left\{u\in L^2([0,1])\colon\langle  u,1\rangle=0\right\}.
\]
Note that $\cP_0(L^2_0([0,1]))\subset L^2_0([0,1])$.
The aim of this section is to prove the next theorem. 

\begin{thm}\label{Th2}
There exists a constant $C>1$ such that for any pair $(\varphi,\psi)$ of functions in 
$C^\theta([0,1])\cap L^2_0([0,1])$, we have
\[
|\langle \cP_0^n \varphi, \psi\rangle| \le C n^{-3/2}\|\varphi\|_{C^\theta}\|\psi\|_{C^\theta}\quad \text{for all $n\ge 1$}.
\]
Moreover, the exponent $-3/2$ is optimal, in that there exist a pair $(\varphi,\psi)$ of functions in  
$C^\theta([0,1])\cap L^2_0([0,1])$ 
and a constant $C(\varphi,\psi)>0$ such that 
\[
|\langle \cP_0^n\varphi,\psi\rangle|\ge C(\varphi,\psi)n^{-3/2}\quad \text{for all }n\ge 1.
\]
\end{thm}

\subsection{Decomposition of the Hilbert space}\label{haar-sec}
It is convenient to
 analyze the action of the reduced Perron-Frobenius operator $\cP_0$ using the Haar wavelets \cite{Wavelet}. Define a function 
$\chi\colon\mathbb{R}\to \mathbb{R}$ by
\[
\chi(x)=\begin{cases}\vspace{1mm}
1&\text{if }\displaystyle{x\in \left[0,\frac{1}{2}\right)},\\ 
-1&\text{if }\displaystyle{x\in \left[\frac{1}{2},1\right)},\\
0&\text{otherwise.}
\end{cases}
\]
For $\ell$, $k\in \mathbb{Z}$ define a function $\chi_{\ell,k}:[0,1]\to\mathbb{R}$ by
\[\chi_{\ell,k}(x)= \chi(2^{\ell-1} x-k).\]
Note that $\chi_{1,0}=\chi|_{[0,1]},$
${\rm supp}(\chi_{\ell,k})=[k2^{-\ell+1},(k+1)2^{-\ell+1}],$
and  
\begin{equation}\label{sum-one}\left|\sum_{0\leq k< 2^{\ell-1}}\chi_{\ell,k}(x)\right|=1\ \text{ for all $x\in[0,1)$.}\end{equation}

The system $\{2^{(\ell-1)/2}\cdot\chi_{\ell,k}\}_{\ell,k\in\mathbb Z}$ of functions is an orthonomal basis of $L^2(\mathbb R)$ 
\cite[Theorem~1.4]{Wavelet}.
Hence the subsystem $\{2^{(\ell-1)/2}\cdot \chi_{\ell,k}\colon\ell\geq 1,\ 0\le k< 2^{\ell-1} \}$ is an orthonormal basis of the space 
$L^2_0([0,1]).$
For each $\ell\ge 1$, let $H_\ell$ denote the $2^\ell$-dimensional closed linear subspace of $L^2_0([0,1])$ that is spanned by  the functions $\{\chi_{\ell,k}\colon 0\le k< 2^{\ell-1}\}$. 
The Hilbert space $L^2_0([0,1])$ is  decomposed into the mutually orthogonal subspaces $H_\ell$: 
\begin{equation}\label{Hdec}
    L^2_0([0,1])=\overline{\bigoplus_{\ell\ge 1} H_\ell}.
\end{equation}
Each $\varphi\in L^2_0([0,1])$ is uniquely written as
\begin{equation}\label{eq:decompH}
    \varphi=\sum_{\ell=1}^\infty \varphi_\ell, \quad\varphi_\ell\in H_\ell.
\end{equation}
Since $\{2^{(\ell-1)/2} \chi_{\ell,k}\colon 
\ell\ge 1, 0\le k< 2^{\ell-1}\}$ is an orthonormal basis of $L^2_0([0,1])$, we have  
\begin{equation}\label{eq:phiell}
\varphi_\ell =\sum_{0\le k< 2^{\ell-1}}\langle \varphi,\chi_{\ell,k}\rangle 2^{\ell-1}\chi_{\ell,k}.
\end{equation}

\subsection{The action of the reduced Perron-Frobenius operator}\label{rPF-sec}
The operator $\cP_0$ is written as the sum
\begin{equation}\label{decomposition-P0}
\cP_0=\cP_\alpha + \cP_\beta,
\end{equation}
where 
\begin{equation}\label{pa}\begin{split}
\cP_\alpha u(x)&=
\begin{cases}\vspace{1mm}
\displaystyle{\frac{1}{2}u(2x)}&\text{ if }\displaystyle{x\in \left[0,\frac{1}{2}\right)},\\
\displaystyle{\frac{1}{2}u(2x-1)}&\text{ if } \displaystyle{x\in \left[\frac{1}{2},1\right],}
\end{cases}\end{split}\end{equation}
and
\begin{equation}\label{pb}\begin{split}\cP_\beta u(x)&=\frac{1}{4}  u\left(\frac{x}{2}\right) +\frac{1}{4}  u\left(\frac{x+1}{2}\right)
\quad \text{ for }x\in [0,1].
\end{split}\end{equation}
The two operators $\cP_\alpha$ and $\cP_\beta$ reflect the contraction in the $x_c$-direction on $\Omega_{\alpha_1}\cup\Omega_{\alpha_2}$ and 
 the expansion in the $x_c$-direction on  $\Omega_{\beta_1}\cup \Omega_{\beta_2}$ respectively.
Note that if $u\in L^1([0,1])$ is strictly increasing (resp. decreasing) then $\cP_0u$ is strictly increasing (resp. decreasing).


    We have
\[
\cP_\alpha(H_\ell)\subset H_{\ell+1}\ \text{ and }\ 
\cP_\beta(H_{\ell+1})\subset H_\ell\ \text{ for all $\ell\geq1$,}
\]
and
\[\cP_\beta(H_1)=\{0\}.\]
See \textsc{Figure~\ref{action1}}.

For the operator norm $\|\cdot\|$ with respect to the $L^\infty$ norm, it is easy to check
\begin{equation}\label{estimatePpm}
     \|\cP_\alpha|_{H_\ell}\|=\frac{1}{2}\quad \text{and}\quad 
    \|\cP_\beta|_{H_{\ell+1}}\|\le \frac{1}{2}\ \text{ for all $\ell\geq1$}.
\end{equation}

\begin{lemma}\label{lm:HolderEst} 
\ 

\begin{itemize}
\item[(a)] If $\varphi\in  C^\theta([0,1])\cap L^2_0([0,1])$, then for all $\ell\geq1$
we have
\[
\|\varphi_{\ell}\|_{L^\infty} \le 2^{-\theta\ell}  \|\varphi\|_{C^\theta}.
\]
\item[(b)]
If $\varphi\in L^2_0([0,1])$
is strictly increasing (resp. decreasing), then 
for all $\ell\geq1$ we have \[ \max_{0\leq k<2^{\ell-1}}\langle \varphi,\chi_{\ell,k}\rangle<0 \quad \left(\text{resp. }\min_{0\leq k<2^{\ell-1}}\langle \varphi,\chi_{\ell,k}\rangle>0\right). \]
\end{itemize}
\end{lemma}
\begin{proof}
For $0\le k<2^{\ell-1}$, we have 
\[
\begin{split}
|\langle \varphi,\chi_{\ell,k}\rangle|&=\left|\int_{2^{-(\ell-1)}k}^{2^{-(\ell-1)}\left(k+\frac{1}{2}\right)}\varphi(x) {\rm d}x - \int^{2^{-(\ell-1)}(k+1)}_{2^{-(\ell-1)}\left(k+\frac{1}{2}\right)}\varphi(x) {\rm d}x\right|\\
&=\left|\int_{2^{-(\ell-1)}k}^{2^{-(\ell-1)}
\left(k+\frac{1}{2}\right)}\big(\varphi(x)-\varphi(x+2^{-\ell})\big) {\rm d}x\right|\\
&\le 2^{-\theta \ell}\|\varphi\|_{C^\theta}\cdot 2^{-\ell},
\end{split}
\]
and thus
\[
\|\varphi_{\ell}\|_{L^\infty}\leq\max_{0\leq k< 2^{\ell-1}}|\langle \varphi,\chi_{\ell,k}\rangle|\cdot 2^{\ell-1} \cdot  \|\chi_{\ell,k}\|_{L^\infty} \le 2^{-\theta\ell}  \|\varphi\|_{C^\theta},\] as required in (a). Since
\[
\langle \varphi,\chi_{\ell,k}\rangle=\int_{2^{-(\ell-1)}k}^{2^{-(\ell-1)}\left(k+\frac{1}{2}\right)}\big(\varphi(x)-\varphi(x+2^{-\ell})\big) {\rm d}x, 
\]
item (b) follows. 
\end{proof}

\begin{lemma}\label{trans-lem}
Let $\varphi\in C^\theta([0,1])\cap L^2_0([0,1])$, and for $n\geq0$ 
write $\cP_0^n\varphi=\sum_{\ell=1}^\infty\varphi_\ell^{(n)}$, $\varphi_\ell^{(n)}\in H_\ell$.
For all $n\geq1$
we have 
\[
\|\varphi^{(n)}_{\ell}\|_{L^\infty}\le\begin{cases}\vspace{1mm}\displaystyle{ \frac{1}{2} \|\varphi^{(n-1)}_{\ell-1}\|_{L^\infty} +
\frac{1}{2} \|\varphi^{(n-1)}_{\ell+1}\|_{L^\infty}}\ &\text{ if }\ell\geq2,\\
\displaystyle{\frac{1}{2} \|\varphi^{(n-1)}_{\ell+1}\|_{L^\infty}}\ &\text{ if }\ell=1.
\end{cases}\]
If we further assume that $\varphi$ is strictly monotone, 
then the above inequalities are actually equalities.
\end{lemma}
\begin{proof}
The claims follows from the decomposition \eqref{decomposition-P0} and the estimates \eqref{estimatePpm}.
\end{proof}

\begin{figure}
\begin{center}
\includegraphics[height=4cm,width=9cm]
{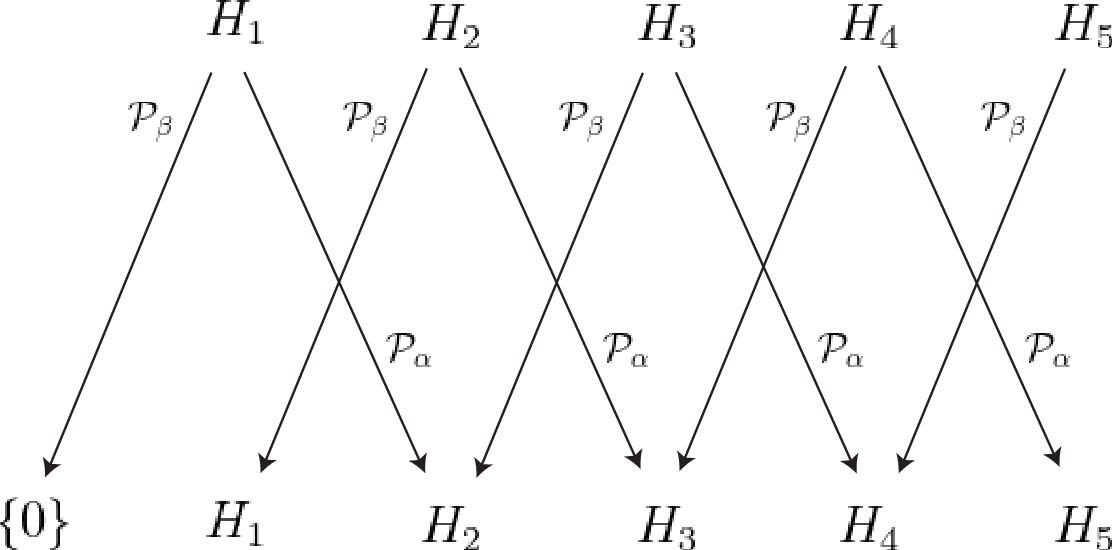}
\caption
{The actions of $\cP_\alpha$ and $\cP_\beta$ on the components $H_\ell$ for $\ell\geq0$. 
}\label{action1}
\end{center}
\end{figure}

\subsection{Comparison with gambler's ruin problem}\label{gamble}
The symmetric simple random walk on $\mathbb N$ with an absorbing wall at $0$ is a well-known model in probability theory, aka `gambler's ruin problem'.
In this model, a particle starts at an integer point in $\mathbb  N$, moves to the right or left with equal probability $1/2$. The particle is absorbed, or disappears from the system once it reaches the point $0$.

  Suppose that the particle is initially distributed at $\ell\in \mathbb{N}$ with probability $q^{(0)}_\ell$:
$\sum_{\ell=1}^\infty q_\ell^{(0)}=1$. 
The probability $q^{(n)}_\ell$ with which the particle is at $\ell\in \mathbb{N}$ at time $n\ge 1$ is given by the recursion formula 
\begin{equation}\label{transition3}
q^{(n)}_\ell=\begin{cases}\vspace{1mm}
\displaystyle{\frac{1}{2}q^{(n-1)}_{\ell-1}+\frac{1}{2}q^{(n-1)}_{\ell+1}}&\quad \text{if $\ell\ge 2$,}\\
\displaystyle{\frac{1}{2}q^{(n-1)}_{\ell+1}}&\quad \text{if $\ell=1$.}
\end{cases}
\end{equation}
The total probability $\sum_{\ell=1}^\infty q_{\ell}^{(n)}$ decreases as $n$ increases due to the absorption at $0$.
For $\ell$, $\ell'\in \mathbb N$ and $n\geq1$, let $p_{\ell\ell'}^{(n)}$ denote the probability with which the particle starts at $\ell$ and reaches $\ell'$ at time $n$ without passing $0$. We have
\begin{equation}\label{prob-eq1}q_\ell^{(n)}=\sum_{\substack{\ell'\geq1\\ |\ell'-\ell|\leq n}}q^{(0)}_{\ell'}p^{(n)}_{\ell'\ell}.\end{equation}

 Taking an appropriate initial probability distribution, one can estimate $\varphi_\ell^{(n)}$
in Lemma~\ref{trans-lem} in terms of $q_\ell^{(n)}$ as follows.
\begin{lemma}\label{dominate-lem}Let $\varphi\in C^\theta([0,1])\cap L^2_0([0,1])$, $\varphi\not\equiv0$, and for $n\geq0$ 
write $\cP_0^n\varphi=\sum_{\ell=1}^\infty\varphi_\ell^{(n)}$, $\varphi_\ell^{(n)}\in H_\ell$. Let 
$(q_\ell^{(n)})_{n=0,\ell=1}^\infty$
be given by \eqref{transition3} and the initial condition
\begin{equation}\label{initial-d}
q^{(0)}_\ell= C_\varphi^{-1}\|\varphi^{(0)}_{\ell}\|_{L^\infty} \quad \text{with }C_\varphi=\sum_{\ell=1}^\infty \|\varphi^{(0)}_{\ell}\|_{L^\infty}\in(0,\infty).
\end{equation}
Then, for all $n\geq0$ we have \begin{equation}\label{dominate-eq}
\|\varphi_\ell^{(n)}\|_{L^\infty}\le C_\varphi  q^{(n)}_\ell\ \text{ for  all }\ell\geq1.
\end{equation}
If we further assume that $\varphi$ is strictly monotone, then the inequality in \eqref{dominate-eq} becomes the equality.\end{lemma}
\begin{proof}
The boundedness
of $C_\varphi$ follows from
 Lemma~\ref{lm:HolderEst}(a).
We prove \eqref{dominate-eq} by induction on $n$. For $n=0$,
\eqref{dominate-eq} is obvious
 from \eqref{initial-d}. Let $k\geq0$ and suppose
  \eqref{dominate-eq} holds for  $n=k$. Comparing the inequalities in Lemma~\ref{trans-lem} with \eqref{transition3}, we get \eqref{dominate-eq} for $n=k+1$. The proof is complete. 
  \end{proof}

In order to estimate $q_\ell^{(n)}$ from above, we compute the transition probability $p_{\ell\ell'}^{(n)}$ and 
give an upper bound when both $\ell$ and $\ell'$ are small relative to $n$. 
\begin{lemma}\label{prob}
For all $\ell,\ell'\geq1$ and $n\geq1$ with $|\ell-\ell'|\leq n$, we have
\[
p_{\ell\ell'}^{(n)}=
\begin{cases}
\displaystyle{2^{-n}\left(\binom{n}{\frac{n-\ell+\ell'}{2} }-\binom{n}{\frac{n-\ell-\ell'}{2}}\right)}
&\quad \text{if $n-\ell+\ell'$ is even,}\\
0&\quad \text{otherwise.}
\end{cases}
\]
\end{lemma}
\begin{proof}
    Although the derivation of this formula is standard in probability theory, we give a proof for completeness. 
    If $n-\ell+\ell'$ is odd, then clearly $p^{(n)}_{\ell\ell'}=0$. Suppose $n-\ell+\ell'$ is even.
    Consider the symmetric simple random walk on $\mathbb Z$. For each $n\geq1$
let $A_n$ denote the set of paths that start from $\ell$ and reach $\ell'$ at time $n$ (possibly passing $0$).
Let $B_n$ denote the set of paths that start from $\ell$ and reach $\ell'$ at time $n$ passing $0$ at least once. Clearly we have
\[2^np^{(n)}_{\ell\ell'}=\#A_n-\#B_n,\] and
\[\#A_n=2^n\binom{n}{\frac{n-\ell+\ell'}{2} }.\]

Let $C_n$ denote the set of paths that start from $\ell$ and reach $-\ell'$ at time $n$. We have
\[\#C_n=2^n\binom{n}{\frac{n-\ell-\ell'}{2} }.\]    
     Let $(s_k)_{k=0}^n$ be an arbitrary path in $B_n$:
    $s_0=\ell$, $s_n=\ell'$ and $s_k=0$ for some $k\in\{1,\ldots,n-1\}$.
Let $\nu=\max\{1\leq k\leq n-1\colon s_k=0\}$ and define another path $(s_k')_{k=0}^n$ by 
\[
s_{k}'=\begin{cases}
s_{k}&\quad \text{for $k\le \nu$,}\\
-s_{k}&\quad \text{for $k>\nu$.}
\end{cases}
\]
Clearly, this correspondence of paths is a bijection from $B_n$ to $C_n$, and so $\#B_n=\#C_n$. Hence the desired equality holds.
\end{proof}

\begin{lemma}\label{lm:num}
For all sufficiently large $n\ge 1$ and $\ell,\ell'\ge 1$ such that $n-\ell+\ell'$ is even and $\max\{\ell,\ell'\}\le n^{1/4}$, we have 
\[
C_0^{-1}n^{-3/2}\ell\ell'\le p_{\ell\ell'}^{(n)}\le C_0n^{-3/2}\ell\ell',
\]   
for some constant $C_0>1$ independent of $n,\ell,\ell'$. 
\end{lemma}
\begin{proof}
We put 
\[
P_{n,\ell,\ell'}=2^{-n}\binom{n}{\frac{n-\ell-\ell'}{2}},
\]
which is the transition probability  from $\ell$ to $-\ell'$ of the symmetric simple random walk on $\mathbb{Z}$.
Using Stirling's approximation for factorials,
we obtain
\begin{equation}\label{lm:num-eq1}
\frac{1}{C\sqrt{n}}\le 2^{-n}\binom{n}{\lfloor\frac{n}{2}+n^{1/4}\rfloor}\le  P_{n,\ell,\ell'}\le 2^{-n}\binom{n}{\lfloor\frac{n}{2}\rfloor }\le  \frac{C}{\sqrt{n}},
\end{equation}
for some constant $C>1$ independent of $n,\ell,\ell'$. 
A direct calculation shows
\begin{align*}
    \frac{\binom{n}{\frac{n-\ell+\ell'}{2}}}
    {\binom{n}{\frac{n-\ell-\ell'}{2}}}&=\frac{n!}{\left(\frac{n-\ell+\ell'}{2}\right)! \left(\frac{n+\ell-\ell'}{2}\right)!}\cdot
\frac{\left(\frac{n-\ell-\ell'}{2}\right)!\left(\frac{n+\ell+\ell'}{2}\right)!}{n!}\\
&=\frac{\left(\frac{n+\ell+\ell'}{2}\right)!}{\left(\frac{n+\ell-\ell'}{2}\right)!}\cdot
\frac{\left(\frac{n-\ell-\ell'}{2}\right)! }
{\left(\frac{n-\ell+\ell'}{2}\right)! }\\
&=\prod_{k=0}^{\ell'-1}\frac{n+\ell+\ell'-2k}{2} \prod_{k'=0}^{\ell'-1}\frac{2}{n-\ell+\ell'-2k'}\\
&=\prod_{k=0}^{\ell'-1}\frac{n+\ell+\ell'-2k}{n-\ell+\ell'-2k}=\prod_{k=0}^{\ell'-1}\frac{1+\frac{\ell+\ell'-2k}{n}}{1-\frac{\ell-\ell'+2k}{n}},
\end{align*}
and thus
\begin{equation}\label{lm:num-eq2}
    2^{-n}\left(\binom{n}{\frac{n-\ell+\ell'}{2}}-\binom{n}{\frac{n-\ell-\ell'}{2}}\right)= P_{n,\ell,\ell'}\left(\prod_{k=0}^{\ell'-1}\frac{1+\frac{\ell+\ell'-2k}{n}}{1-\frac{\ell-\ell'+2k}{n}}-1\right).
\end{equation}
By the assumption $\max\{\ell,\ell'\}\leq n^{1/4}$ and the simple estimate
$1+x\leq e^{x}$ when $|x|$ is small,
we have
\[\begin{split}\prod_{k=0}^{\ell'-1}\frac{1+\frac{\ell+\ell'-2k}{n}}{1-\frac{\ell-\ell'+2k}{n}}&=\prod_{k=0}^{\ell'-1}\left(1+\frac{\frac{2\ell }{n}}{1-\frac{\ell-\ell'+2k}{n}}\right)\leq\prod_{k=0}^{\ell'-1}\exp\left(\frac{\frac{2\ell }{n}}{1-\frac{\ell-\ell'+2k}{n}}\right)\\
&\leq\prod_{k=0}^{\ell'-1}\exp\left(\frac{3\ell }{n}\right)=\exp\left(\frac{3\ell\ell'}{n}\right).\end{split}\]
Plugging this estimate into the right-hand side of \eqref{lm:num-eq2}, and then using
\eqref{lm:num-eq1} and another simple estimate
$e^x\leq 1+x+x^2$ when $|x|$ is small, we obtain 
\begin{align*}
2^{-n}\left(\binom{n}{\frac{n-\ell+\ell'}{2}}-\binom{n}{\frac{n-\ell-\ell'}{2}}\right)
&\le P_{n,\ell,\ell'}  \left(\exp\left(\frac{3\ell\ell'}{n}\right)-1\right)\\
    &\le 4C n^{-1/2}\cdot \frac{\ell\ell'}{n}=4Cn^{-3/2}\ell\ell'.
\end{align*}

Similarly, by $\max\{\ell,\ell'\}\leq n^{1/4}$ and $ 1+x\geq e^{x/2}$ when $x>0$ is small, we have
\[\prod_{k=0}^{\ell'-1}\frac{1+\frac{\ell+\ell'-2k}{n}}{1-\frac{\ell-\ell'+2k}{n}}\geq\exp\left(\frac{\ell\ell'}{n}\right).\]
Plugging this estimate into the right-hand side of \eqref{lm:num-eq2},
and then using \eqref{lm:num-eq1} and $e^{x}\geq1+x$ when $|x|$ is small, we obtain
\begin{align*}
2^{-n}\left(\binom{n}{\frac{n-\ell+\ell'}{2}}-\binom{n}{\frac{n-\ell-\ell'}{2}}\right)
&\ge P_{n,\ell,\ell'} \left(\exp\left(\frac{\ell\ell'}{n}\right)-1\right)\\
    &\ge C^{-1} n^{-1/2}\cdot \frac{\ell\ell'}{n}=C^{-1}\ell\ell' n^{-3/2}.
\end{align*}
Taking $C_0=4C$ yields the desired inequalities.
\end{proof}

\subsection{Polynomial lower bounds}\label{low-sec}
The next lemma establishes the lower bound in Theorem~\ref{Th2}.
\begin{lemma}\label{lm:lower} If  $\varphi,\psi\in  C^\theta([0,1])\cap L^2_0([0,1])$ 
are strictly monotone, then there exists a constant $C(\varphi,\psi)>0$ such that 
\[
|\langle \cP_0^n\varphi,\psi\rangle|\ge C(\varphi,\psi)n^{-3/2}\quad \text{for all }n\ge 1.
\]
\end{lemma}
\begin{proof}
For $n\geq0$ 
write $\cP_0^n\varphi=\sum_{\ell=1}^\infty\varphi_\ell^{(n)}$, $\varphi_\ell^{(n)}\in H_\ell$ and
$\psi=\sum_{\ell=1}^\infty\psi_\ell$,
 $\psi_\ell\in H_\ell$.
Let $(q_\ell^{(n)})_{n=0,\ell=1}^\infty$ be given by \eqref{transition3} and the initial condition
\[
q^{(0)}_\ell= C_\varphi^{-1}\|\varphi^{(0)}_{\ell}\|_{L^\infty} \quad \text{with }C_\varphi=\sum_{\ell=1}^\infty \|\varphi^{(0)}_\ell\|_{L^\infty}\in(0,\infty).
\]
For all $n\geq1$ we have  
\begin{align*}
|\langle \cP_0^n\varphi,\psi\rangle|&=\left|\sum_{\ell=1}^\infty\langle \varphi_\ell^{(n)},\psi_\ell\rangle\right|\geq  \left|\langle \varphi_1^{(n)},\psi_1\rangle\right| \ge \|\varphi_{1}^{(n)}\|_{L^\infty}\|\psi_{1}\|_{L^\infty}\geq C_\varphi q_1^{(n)} \|\psi_{1}\|_{L^\infty}.
\end{align*}
 The first inequality is because 
 $\varphi^{(n)}_\ell$ $(\ell\geq1)$ are all positive or all negative and $\psi_\ell$ $(\ell\geq1)$ are all positive or all negative,
which follows from  
 the strict monotonicity of $\cP_0^n\varphi$ and that of $\psi$ 
 as in Lemma~\ref{lm:HolderEst}(b). The second inequality is because  $\varphi_1^{(n)}$ and $\psi_1$ are both non-zero multiples of $\chi_{1,0}$ and
$\|\varphi_{1}^{(n)}\|_{L^\infty}= C_\varphi q_{1}^{(n)}$ by the last claim of Lemma~\ref{dominate-lem}.
Finally, from \eqref{prob-eq1} and Lemma~\ref{lm:num} there exists a constant $c=c(\varphi)>0$ such that
$q_1^{(n)}\ge c n^{-3/2}$
for every $n\geq1$.
Hence the desired inequality holds.
\end{proof}

\subsection{Proof of Theorem~\ref{Th2}}\label{pfthm-sec}
Let $\varphi,\psi\in  C^\theta([0,1])\cap L^2_0([0,1])$. For $n\geq0$ 
write $\cP_0^n\varphi=\sum_{\ell=1}^\infty\varphi_\ell^{(n)}$, $\varphi_\ell^{(n)}\in H_\ell$ and
$\psi=\sum_{\ell=1}^\infty\psi_\ell$, $\psi_\ell\in H_\ell$.
Let $(q_\ell^{(n)})_{n=0,\ell=1}^\infty$ be given by \eqref{transition3} and the initial condition \eqref{initial-d}. 
We obtain
\begin{align*}
|\langle \cP_0^n \varphi,\psi\rangle|
=& \left|\sum_{\ell=1}^\infty \langle \varphi^{(n)}_\ell, \psi_{\ell}\rangle \right|
\le \sum_{\ell=1}^\infty \| \varphi^{(n)}_\ell\|_{L^\infty} \|\psi_{\ell}\|_{L^\infty}\\
\leq& C_\varphi \|\psi\|_{C^\theta}\sum_{\ell=1}^\infty q_\ell^{(n)}2^{-\theta\ell}=C_\varphi\|\psi\|_{C^\theta}\sum_{\ell=1}^\infty  
\sum_{\substack{\ell'\geq1\\ |\ell'-\ell|\leq n}}
q^{(0)}_{\ell'}p^{(n)}_{\ell',\ell}2^{-\theta\ell}\\
\leq& C_\varphi\|\psi\|_{C^\theta}\left(\sum_{\max\{\ell,\ell'\}<n^{1/4}} \sum_{\substack{\ell'\geq1\\ |\ell'-\ell|\leq n}}
q^{(0)}_{\ell'}p^{(n)}_{\ell',\ell}2^{-\theta\ell}+\sum_{\ell\geq n^{1/4}}2^{-\theta\ell}\right)\\
\leq&\|\varphi\|_{C^\theta}\|\psi\|_{C^\theta}\left(\sum_{\max\{\ell,\ell'\}<n^{1/4}} \sum_{\substack{\ell'\geq1\\ |\ell'-\ell|\leq n}}C_0n^{-3/2}\ell\ell'2^{-\theta(\ell+\ell')}+\sum_{\ell\geq n^{1/4}}2^{-\theta\ell}\right)\\
\leq& C n^{-3/2} \|\varphi\|_{C^\theta} \|\psi\|_{C^\theta},
\end{align*}
provided that the constant $C$ is sufficiently large
depending only on $\theta$.
For the second inequality we have used
  Lemma~\ref{lm:HolderEst}(a) to bound $\|\psi_\ell\|_{L^\infty}$, and Lemma~\ref{dominate-lem} to bound $\|\varphi_\ell^{(n)}\|_{L^\infty}$. For the fourth inequality we have used Lemma~\ref{lm:num} and
 $q^{(0)}_{\ell'}\leq C_\varphi^{-1}2^{-\theta\ell'}\|\varphi\|_{C^\theta}$ that follows from \eqref{initial-d} and Lemma~\ref{lm:HolderEst}(a).
 This upper bound and
the lower bound in Lemma~\ref{lm:lower} complete the proof of Theorem~\ref{Th2}. 
\qed

\section{Proof of the Main Theorem in the case $M=2$}\label{sec:pf}
In this section we complete the proof of the Main Theorem for the case $M=2$ using the original Perron-Frobenius operator $\cP$ in \eqref{PF}. We take a pair $(\varphi,\psi)$ of functions in $C^\theta([0,1]^3)$ and estimate their correlation
\[
\mathrm{Cor}(\varphi,\psi\circ f^n)=
\left|\langle \cP^n \varphi,\psi\rangle - \langle \varphi,1\rangle\langle \psi,1\rangle\right|.
\]
Since the correlation does not change even if we subtract constant functions from $\varphi$ and $\psi$, we may 
assume $\langle \varphi,1\rangle=\langle \psi,1\rangle=0$ without loss of generality.

The argument below is mostly parallel to that of the proof of Theorem~\ref{Th2} in Section~$2$, but has some important differences. The obvious difference is that  $\cP$ acts on the space of functions on $[0,1]^3$.  
In Section~\ref{decomp-sec} we introduce a few notations that we will work with. One important difference from Section~$2$ is that instead of the
$L^\infty$ norm we work with an anisotropic norm $\|\cdot\|_{\theta/2}$ introduced 
in Section~\ref{norm-section}. In
Section~\ref{action-sec} and Section~\ref{adecomp} we analyze the action of $\cP$,
and complete the proof of the Main Theorem in Section~\ref{pf-main}.


\subsection{A decomposition of the Hilbert space}\label{decomp-sec}
In order to trace the argument in Section~$2$, we regard functions on $[0,1]^3$ as 
tensor products of functions of 
 two variables $(x_u,x_s)\in [0,1]^2$ and
functions of one variable $x_c\in [0,1]$. For $u\in L^2([0,1]^2)$ and $v\in L^2([0,1])$, we write $u\otimes v\in L^2([0,1]^3)$ for the function 
\begin{equation}\label{tensor}
u\otimes v(x_u,x_c,x_s)=u(x_u,x_s)v(x_c).
\end{equation}
For $d=2,3$ we set 
\[
L^2_0([0,1]^d)=\{u\in L^2([0,1]^d \colon \langle u,1\rangle =0\}.
\]
We are going to consider the operator $\cP$ acting on the Hilbert space $L^2_0([0,1]^3)$.

We write $\chi_0$ for the constant function on $[0,1]$ that takes value $1$.
Recall that the collection of functions $\{\chi_{0}\}\cup\{\chi_{\ell,k}\colon\ell\ge 1, 0\le k< 2^{\ell-1}\}$ is an  orthogonal basis of $L^2([0,1])$. 
For each $\ell\geq0$, let $\hat H_\ell$ denote the infinite dimensional subspace of $L^2_0([0,1]^3)$ given by 
\begin{align*}
\hat H_\ell&=\left\{\sum_{0\leq k<2^{\ell-1}} u_{\ell,k}\otimes \chi_{\ell,k}\colon u_{\ell,k}\in L^2_0([0,1]^2)\right\}\ \text{ if $\ell\geq1$,}
\intertext{and
}
\hat H_0&=\{ u_{0}\otimes \chi_{0}\colon u_{0} \in L^2_0([0,1]^2)\}.
\end{align*}

The space $L^2_0([0,1]^3)$ is decomposed into the mutually orthogonal subspaces $\hat H_\ell$: 
\[
L^2_0([0,1]^3)=\overline{\bigoplus_{\ell\ge 0} \hat H_\ell}.
\]
Unlike the space $L_0^2([0,1])$ in \eqref{Hdec}
we now have the component $\hat H_0$. This is the main point to be worked out 
in Section~\ref{adecomp}.

 Under the notation \eqref{tensor},
  each function $\varphi\in L_0^2([0,1]^3)$ is uniquely written as 
\begin{equation}\label{decomp}
\varphi=\varphi_{0,0}\otimes \chi_0+\sum_{\ell=1}^{\infty}\sum_{0\leq k<2^{\ell-1} } \varphi_{\ell,k}\otimes  \chi_{\ell,k},
\end{equation}
so that $\varphi_{0,0}\otimes \chi_0\in \hat H_0$ and $
\sum_{0\leq k<2^{\ell-1} } \varphi_{\ell,k}\otimes  \chi_{\ell,k}\in \hat H_\ell$,  where
\begin{align*}
\varphi_{0,0}(x_u,x_s)&=\int_{[0,1]} \varphi(x_u,t,x_s) {\rm d}t\ \text{ and }\\
\varphi_{\ell,k}(x_u,x_s)&=2^{\ell-1}\int_{[0,1]} \varphi(x_u,t,x_s) \chi_{\ell,k}(t) {\rm d}t.
\end{align*}
The same reasoning as in the proof of Lemma~\ref{lm:HolderEst}(a) yields the estimates
\[
\|\varphi_{0,0}\|_{L^\infty}\le  \|\varphi\|_{C^\theta},
\]
and
\[
\sup_{0\leq k<2^{\ell-1}}\|\varphi_{\ell,k}\|_{L^\infty}\le  2^{-\theta\ell} \|\varphi\|_{C^\theta}\ \text{ for all }\ell\geq1.
\]

\subsection{
Anisotropic norms}\label{norm-section}
For an element $u(x_u,x_s)$ of $L^2([0,1]^2)$, we write $u(\cdot,x_s)$ for the function $x_u\mapsto u(x_u,x_s)$ with the second variable $x_s$ fixed.
Let $0<\theta<1$. 
We define $\hat{C}^\theta([0,1]^2)\subset L^2([0,1]^2)$ as the subspace of functions $u(x_u,x_s)\in L^\infty([0,1]^2)$ such that the function $u(\cdot,x_s)$ belongs to $C^\theta([0,1])$ for all $x_s\in [0,1]$. For a function 
$\hat{C}^\theta([0,1]^2)$ and $x_s\in [0,1]$, we define
\[
|u(\cdot,x_s)|_{C^{\theta}}=
\sup_{\substack{x_u,x'_u\in[0,1]\\x_u\neq x'_u }} 
\frac{|u(x_u,x_s)-u(x'_u,x_s)|}{|x_u-x'_u|^{\theta}}.
\]

In the same spirit as that of the anisotropic norms in \cite{BGL02}, we introduce the following norm on the space $\hat{C}^\theta([0,1]^2)$:
\begin{equation}\label{norm-def}
\|u\|_{\theta/2}= |u|_\infty+|u|_{\theta/2}
\end{equation}
where 
\[
|u|_{\infty}=\int_{[0,1]}\|u(\cdot,x_s)\|_{L^\infty} {\rm d}x_s
\ \text{ and } \ 
|u|_\theta=\int_{[0,1]}|u(\cdot,x_s)|_{C^\theta} {\rm d}x_s.
\]
Note that the subscript in \eqref{norm-def} is $\theta/2$ and not $\theta$, which is intentional. 
Clearly we have 
\begin{equation}\label{eq:ellone}
    \|u\|_{L^1}\le |u|_{\infty}\le \|u\|_{\theta/2}.
\end{equation}

For each $\ell\geq 1$, we define a subspace $\hat H_{\ell,\hat{C}^{\theta}}$ of $\hat H_\ell$ by
\begin{align*}
\hat H_{\ell,\hat{C}^{\theta}}&=\left\{\sum_{0\leq k<2^{\ell-1}}u_{\ell,k}\otimes\chi_{\ell,k}\colon u_{\ell,k}\in \hat{C}^\theta([0,1]^2)\right\}.
\intertext{
For the case $\ell=0$, we set }
\hat H_{0,\hat{C}^{\theta}}&=\left\{u_{0,0}\otimes\chi_{0}\colon u_{0,0}\in \hat{C}^\theta([0,1]^2)\cap L^2_0([0,1]^2)\right\}.
\end{align*}
 


By definition, each $u\in \hat H_{\ell, \hat{C}^\theta}$ is written in the form 
\begin{align*}
u&=\sum_{0\leq k<2^{\ell-1}}u_{\ell,k}\otimes  \chi_{\ell,k} \quad \text{ if } \ell\geq1
\intertext{and}
u&=u_{0,0}\otimes \chi_{0} \quad\text{ if } \ell=0
\end{align*}
with $u_{\ell,k}\in \hat{C}^{\theta}([0,1]^2)$.
We define a norm $\|\cdot\|_{\theta/2}^{(\ell)}$ on $\hat H_{\ell, \hat{C}^\theta}$ by
\[\label{eq:maxnorm}\|u\|_{\theta/2}^{(\ell)}=\begin{cases}\max_{0\le k<2^{\ell-1}}\|u_{\ell,k}\|_{\theta/2}&\quad\text{ if } \ell\geq1,\\
\|u_{0,0}\|_{\theta/2}&\quad\text{ if } \ell=0.\end{cases}\]

\begin{lemma}\label{norm-new}
For $\varphi\in  C^\theta([0,1]^3)$, we write
$\varphi=\sum_{\ell=0}^\infty\varphi_\ell$ with $\varphi_\ell\in\hat H_\ell$.  Then we have 
\[\|\varphi_\ell\|_{L^1}\leq \|\varphi_\ell\|_{\theta/2}^{(\ell)}\ \text{ for all }\ell\geq0.\]\end{lemma}
\begin{proof}
We have $\varphi_\ell=\sum_{0\leq k<2^{\ell-1}}\varphi_{\ell,k}\otimes\chi_{\ell,k}$ for $\ell\geq1$ and
$\varphi_0=\varphi_{0,0}\otimes\chi_0$.
For $\ell\geq1$, using \eqref{sum-one} and \eqref{eq:ellone}, we see
\[\|\varphi_\ell\|_{L^1}\leq \max_{0\leq k<2^{\ell-1}}\|\varphi_{\ell,k}\|_{L^1}\leq \max_{0\leq k<2^{\ell-1}}\|\varphi_{\ell,k}\|_{\theta/2}
=\|\varphi_{\ell}\|_{\theta/2}^{(\ell)}.\]
Using \eqref{eq:ellone} we also have
$\|\varphi_0\|_{L^1}\leq \|\varphi_{0,0}\|_{L^1}
= \|\varphi_{0,0}\|_{\theta/2}^{(0)}$.
\end{proof}

Below we prove two technical lemmas on the new norm $\|\cdot\|_{\theta/2}^{(\ell)}$
that will be used later.
The next lemma is an analogue of Lemma~\ref{lm:HolderEst}(a). \begin{lemma}\label{norm-bound-lem}
Let $\varphi\in  C^\theta([0,1]^3)\cap L^2_0([0,1]^3)$ and write
$\varphi=\sum_{\ell=0}^\infty\varphi_\ell$, $\varphi_\ell\in\hat H_\ell$.  Then we have 
\[
\|\varphi_{\ell}\|_{\theta/2}^{(\ell)}\le 2\cdot 2^{-(\theta/2) \ell}\|\varphi\|_{C^\theta}\ \text{ for all }\ell\geq0.
\]
\end{lemma}
\begin{proof}
 Let $\varphi_{0,0}$ and  $\varphi_{\ell,k}$ be the 
coefficients in the expansion \eqref{decomp} 
of $\varphi$. It suffices to show that
 \[
\|\varphi_{0,0}\|_{\theta/2}\le 2 \|\varphi\|_{C^\theta}
\]
and
\[
\max_{0\leq k<2^{\ell-1}}\|\varphi_{\ell,k}\|_{\theta/2}\le 2\cdot 2^{-(\theta/2) \ell}\|\varphi\|_{C^\theta}\ \text{ for all }\ell\geq1.
\]

Let $\ell\ge 1$ and $0\le k<2^{\ell-1}$. An argument similar to the one in the proof of Lemma~\ref{lm:HolderEst} gives 
\[
|\varphi_{\ell,k}|_{\infty}\le 2^{-\theta\ell}|\varphi|_{C^\theta}. 
\]
Observe that    
\[
|\varphi_{\ell,k}|_{\theta/2}\le \sup_{x_s\in[0,1]} 
\sup_{\substack{x_u,x_u'\in[0,1] \\ x_u\neq x'_u}}
\frac{\left|\int_{[0,1]} (\varphi(x_u,x_c,x_s) -\varphi(x'_u,x_c,x_s))\cdot 2^{\ell-1} \chi_{\ell,k}(x_c){\rm d}x_c\right|}{|x_u-x'_u|^{\theta/2}}.
\]
From an obvious estimate
\[
|\varphi(x_u,x_c,x_s) -\varphi(x'_u,x_c,x_s)|\le |\varphi|_{C^\theta}\cdot |x_u-x'_u|^{\theta},
\]
we have 
 \[
 \left|\int_{[0,1]} (\varphi(x_u,x_c,x_s) -\varphi(x'_u,x_c,x_s))\cdot 2^{\ell-1} \chi_{\ell,k}(x_c){\rm d}x_c\right|\le |\varphi|_{C^\theta}\cdot  |x_u-x'_u|^{\theta}. 
 \]
 Besides, by an estimate similar to that in the proof of Lemma~\ref{lm:HolderEst}(a), we have 
 \[
 \left|\int_{[0,1]} (\varphi(x_u,x_c,x_s) -\varphi(x'_u,x_c,x_s))\cdot 2^{\ell-1} \chi_{\ell,k}(x_c){\rm d}x_c\right|\le  2^{-\theta \ell}|\varphi|_{C^\theta}. 
 \]
 Therefore we obtain   
\begin{align*}
    &\left|\int_{[0,1]} (\varphi(x_u,x_c,x_s) -\varphi(x'_u,x_c,x_s))\cdot 2^\ell \chi_{\ell,k}(x_c){\rm d}x_c\right|\\
    &\qquad \le \min\left\{ 2^{-\theta \ell}, |x_u-x'_u|^\theta\right\}\cdot|\varphi|_{C^\theta} \le  2^{-\theta\ell/2} |x_u-x'_u|^{\theta/2}\cdot|\varphi|_{C^\theta}.
\end{align*}
We conclude $\|\varphi_{\ell,k}\|_{\theta/2}\le 2\cdot 2^{-\theta\ell/2}|\varphi|_{C^\theta}$ as required. We obtain the same estimate on $\varphi_{0,0}$ by a similar argument. 
\end{proof}

\begin{lemma}\label{lm:rect} 
Let $I$ be a non-degenerate closed subinterval of $[0,1]$. Put 
$R=I\times[0,1]$ and suppose that $g\colon R\to\mathbb R^2$ is an invertible affine map of the form $g(x_u,x_s)=(\alpha x_u+p_u, \beta x_s+p_s)$ with $0<\beta<1<\alpha$ such that $g(R)=[0,1]\times J$ with $J\subset [0,1]$. For a function $u\in \hat{C}^\theta([0,1]^2)$, define
$\tilde u\colon [0,1]^2\to\mathbb R$ by \[
\tilde{u}(x_u,x_s)=\begin{cases}
    u\circ g^{-1}(x_u,x_s),&\text{ if $(x_u,x_s)\in g(R)$,}\\
    0,&\text{ otherwise.}
\end{cases}
\]
Then we have $\tilde{u}\in\hat{C}^\theta([0,1]^2)$ and 
\[
\|\tilde{u}\|_{\theta/2}\le |\beta| \|u\|_{\theta/2}.
\]
\end{lemma}
\begin{proof}
For  $(x_u,x_s)\in g(R)$, we set 
$g^{-1}(x_u,x_s)=(x_u',x_s')$.
By definition we have
\begin{align*}
|\tilde{u}|_\infty&=\int_J \|\tilde{u}(\cdot,x_s)\|_{L^\infty} {\rm d}x_s
\leq
\int_{[0,1]} \|u(\cdot,x_s')\|_{L^\infty} \beta{\rm d}x_s'
\le |\beta| |u|_\infty,
\end{align*}
and similarly 
\begin{align*}
|\tilde{u}|_{\theta/2}&=\int_J |\tilde{u}(\cdot,x_s)|_{C^{\theta/2}} {\rm d}x_s\\&\le  \int_{[0,1]} |\alpha|^{-\theta/2} |u(\cdot,x_s') |_{C^{\theta/2}} \beta{\rm d}x'_s
\le |\alpha|^{-\theta/2}|\beta|\, |u|_{\theta/2}\leq|\beta|\, |u|_{\theta/2}.
\end{align*}
These estimates yield the claim of the lemma.
\end{proof}

\subsection{The action of the Perron-Frobenius operator }\label{action-sec}
We give a description of the operator $\cP$ 
mostly parallel to that of $\cP_0$ in Subsection \ref{rPF-sec}. The precise formulae below does not look very simple but not difficult to check referring \textsc{Figure~\ref{hetero3}} and \textsc{Figure~\ref{action2}}.

We write the operator $\cP$ as the sum  
\begin{equation}\label{decomposition-P} 
\cP=\hat{\cP}_\alpha +\hat{\cP}_\beta,\end{equation}
where
$\hat{\cP}_\alpha$ and $\hat{\cP}_\beta$ are given as follows on each $\hat H_\ell$ with $\ell \ge 0$:
\begin{itemize}
    \item[(A)] On $\hat H_\ell$ with $\ell\ge 2$, we have 
\begin{align*}
\hat{\cP}_\alpha (u_{\ell,k}\otimes \chi_{\ell,k})(x_u,x_c,x_s)&= u_{\ell,k}
\left(\frac{x_u}{4},2x_s\right) \chi_{\ell+1,k}(x_c)\\
&\quad +
u_{\ell,k}\left(\frac{x_u+1}{4},2x_s\right)\chi_{\ell+1,k+2^{\ell-1}}(x_c),
\end{align*}
and  
\begin{align*}
&\hat{\cP}_\beta (u_{\ell,k}\otimes \chi_{\ell,k})(x_u,x_c,x_s)\\
&\qquad=
\begin{cases}\vspace{1mm}\displaystyle{u_{\ell,k}\left(\frac{x_u+1}{2},4x_s-2\right)  \chi_{\ell-1,k}(x_c)}& \text{if }0\le k<2^{\ell-2},\\
\displaystyle{u_{\ell,k}\left(\frac{x_u+1}{2},4x_s-3\right)  \chi_{\ell-1,k-2^{\ell-2}}(x_c)}& \text{if }
2^{\ell-2}\le k<2^{\ell-1}.
\end{cases}
\end{align*}

\item[(B)] On $\hat H_1$, we have
\begin{align*}
\hat{\cP}_\alpha (u_{1,0}\otimes\chi_{1,0})(x_u,x_c,x_s)&=u_{1,0}\left(\frac{x_u}{4},2x_s\right) \chi_{2,0}(x_c)\\
&\quad +
u_{1,0}\left(\frac{x_u+1}{4},2x_s\right) \chi_{2,1}(x_c)
\end{align*}
and 
\begin{align*}
\hat{\cP}_\beta (u_{1,0}\otimes\chi_{1,0})(x_u,x_c,x_s)&=
u_{1,0}\left(\frac{x_u+1}{2},4x_s-2\right)\chi_{0}(x_c)\\
&\quad -
u_{1,0}\left(\frac{x_u+1}{2},4x_s-3\right) \chi_{0}(x_c).
\end{align*}

\item[(C)] On $\hat H_0$, we have
\begin{align*}
&\hat{\cP}_\alpha(u_{0}\otimes\chi_{0})(x_u,x_c,x_s) =\frac{1}{2}\left(u_{0}\left(\frac{x_u}{4},2x_s\right)-u_{0}\left(\frac{x_u+1}{4},2x_s\right)\right)\chi_{1,0}(x_c)
\end{align*}
and
\begin{align*}
&\hat{\cP}_\beta(u_{0}\otimes\chi_{0})(x_u,x_c,x_s)\\
&\quad = 
\left(u_{0}\left(\frac{x_u+1}{2},4x_s-2\right)+u_{0}\left(\frac{x_u+1}{2},4x_s-3\right)\right) \chi_{0}(x_c).
\end{align*}
\end{itemize}
It is easy to check the following relation illustrated in \textsc{Figure~\ref{action2}}:
\begin{equation}\hat{\cP}_\alpha(\hat H_\ell)\subset \hat H_{\ell+1}\ \text{ and }\  \hat{\cP}_\beta(\hat H_\ell)\subset \hat H_{\max\{0,\ell-1\}}\ \text{ for all }\ell\geq0. \end{equation}

\begin{figure}
\begin{center}
\includegraphics[height=4cm,width=9cm]
{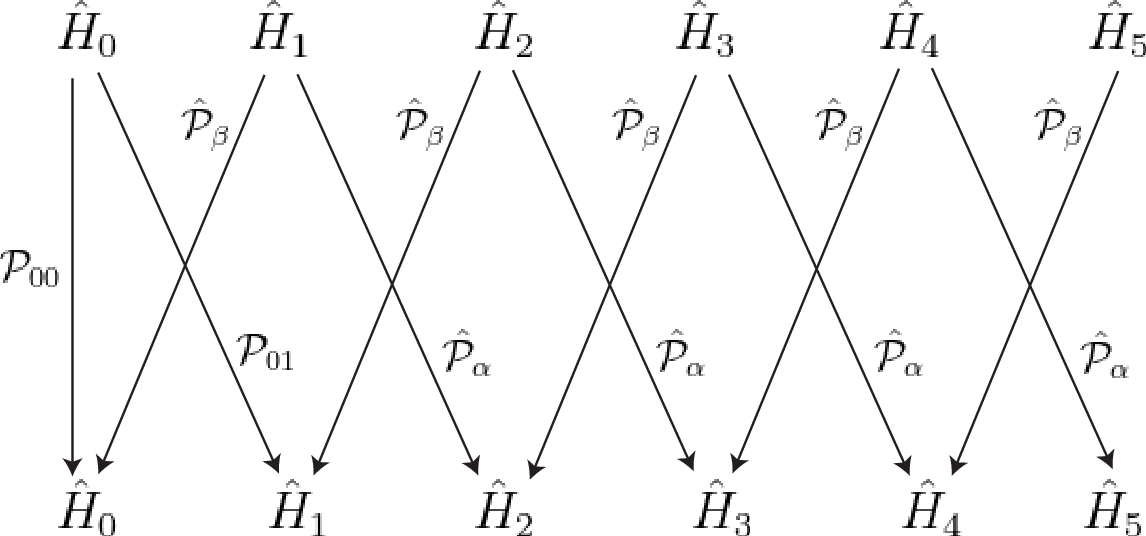}
\caption
{The actions of $\hat{\cP}_\alpha$, $\hat{\cP}_\beta$, $\cP_{01}$, $\cP_{00}$ on the components
$\hat H_\ell$ for $\ell\ge 0$. 
We have $\cP_{01}|_{\hat H_0}=\cP_\alpha|_{\hat H_0}$ and $\cP_{00}|_{\hat H_0}=\cP_\beta|_{\hat H_0}$.}\label{action2}
\end{center}
\end{figure}


As a direct consequence of the above explicit expressions of $\hat{\cP}_{\alpha}$, $\hat{\cP}_\beta$ and Lemma~\ref{lm:rect},
for the operator norm $\|\cdot\|$ with respect to the norm $\|\cdot\|_{\theta/2}^{(\ell)}$ on $\hat H_{\ell,\hat{C}^\theta}$, we have
\begin{equation}\label{lm:Ppm}
\left\|\hat{\cP}_{\alpha}|_{\hat H_{\ell,\hat{C}^\theta}}\right\|\le \frac{1}{2} \ \text{ and }\  \left\|\hat{\cP}_{\beta}|_{\hat H_{\ell,\hat{C}^\theta}}\right\|\le \frac{1}{2}\ \text{ for all $\ell\geq0$.}
\end{equation}

The next lemma is an analogue of Lemma~\ref{trans-lem}.

\begin{lemma}\label{cor:trans}
Let $\varphi\in C^\theta([0,1]^3)\cap L^2_0([0,1]^3)$, and 
for $n\ge 0$ write $\cP_*^n \varphi=\sum_{\ell=1}^\infty \varphi_\ell^{(n)}$ with
$\varphi_\ell^{(n)}\in \hat H_\ell$.
Then, for all $n\geq 1$ and all $\ell\ge 1$ we have  $\varphi_0^{(n)}=0$ and $\varphi_\ell^{(n)}\in \hat{H}_{\ell,\hat{C}^\theta}$, and
\[\|\varphi^{(n)}_{\ell}\|_{\theta/2}^{(\ell)}\le\begin{cases}\vspace{1mm}\displaystyle{\frac{1}{2} \|\varphi^{(n-1)}_{\ell-1}\|_{\theta/2}^{(\ell-1)} +
\frac{1}{2} \|\varphi^{(n-1)}_{\ell+1}\|_{\theta/2}^{(\ell+1)}}\ &\text{ if }\ell\geq2,\\
\displaystyle{\frac{1}{2} \|\varphi^{(n-1)}_{\ell+1}\|_{\theta/2}^{(1)}}
\ & \text{ if }\ell=1.\end{cases}\]
\end{lemma}
\begin{proof}The claims follow from Lemma~\ref{lm:rect},  the decomposition \eqref{decomposition-P} and the estimates \eqref{lm:Ppm}.\end{proof}


\subsection{Discrepancies between the reduced and the original operators}\label{adecomp}
The reduced Perron-Frobenius operator $\cP_0$
acts on the space  $L_0^2([0,1])=\overline{\bigoplus_{\ell\ge 1} H_{\ell}}$, while
the original Perron-Frobenius operator $\cP$ acts on the space  $L_0^2([0,1]^3)=\overline{\bigoplus_{\ell\ge 0} \hat H_{\ell}}$.
In order to analyze the action of $\cP$ we must take the transitions $\hat H_0\to \hat H_0$, 
$\hat H_0\to\hat H_1$, $\hat H_1\to \hat H_0$ into account:
Compare \textsc{Figure}~\ref{action1} and \textsc{Figure}~\ref{action2}.

Let $\pi_0\colon L_0^2([0,1]^3)\to \hat H_0$ denote the natural projection to the component $\hat{H}_0$. Let $I$ denote the identity operator on $L_0^2([0,1]^3)$. The operator $\cP$ is written as the sum  
\[
\cP=\cP_*+\cP_{10}+\cP_{01}+\cP_{00},
\]
where
\begin{align*}
    \cP_*&=(I-\pi_0)\circ\cP\circ (I-\pi_0),\\
    \cP_{01}&=(I-\pi_0)\circ\cP\circ\pi_0,\\
    \cP_{10}&=\pi_0\circ \cP\circ (I-\pi_0),\\
    \cP_{00}&=\pi_0\circ \cP\circ \pi_0.
\end{align*}
Note that
$\cP_{01}|_{\hat H_0}=\cP_\alpha|_{\hat H_0}$ and $\cP_{00}|_{\hat H_0}=\cP_\beta|_{\hat H_0}$ as in \textsc{Figure}~\ref{action2}. 

For the part $\cP_*$, we may proceed with the argument parallel to that for  $\cP_0$ in Section~\ref{sec2}. Indeed 
the next lemma is an analogue of Lemma~\ref{dominate-lem}.
\begin{lemma}\label{dominate-lem-2}Let $\varphi\in C^\theta([0,1]^3)\cap L^2_0([0,1]^3)$ and suppose $\varphi\not\equiv0$. We write 
\[
\varphi=\sum_{\ell=0}^\infty\varphi_\ell^{(0)}\quad\text{and}\quad 
\cP_*^n\varphi=\sum_{\ell=1}^\infty\varphi_\ell^{(n)}\quad\text{for $n\ge 1$}
\]
with $\varphi_\ell^{(n)}\in \hat H_{\ell,\hat{C}^\theta}$.
Let  $(q_\ell^{(n)})_{n=0,\ell=1}^\infty$ be given by \eqref{transition3} and the initial condition
\begin{equation}\label{initial-d-2}
q^{(0)}_\ell= \hat C_\varphi^{-1}\|\varphi^{(0)}_{\ell}\|_{\theta/2}^{(\ell)}\ \text{ and }\ \hat{C}_\varphi=\sum_{\ell=1}^\infty \|\varphi^{(0)}_\ell\|_{\theta/2}^{(\ell)}\in(0,\infty).\end{equation}
Then, for all $n\geq 0$ and $\ell \ge 1$, we have \[\label{dominate-eq-2}
\|\varphi_\ell^{(n)}\|_{\theta/2}^{(\ell)}\le \hat C_\varphi  q^{(n)}_\ell.
\]\end{lemma}
\begin{proof}
The boundedness of $\hat C_\varphi$ follows from Lemma~\ref{norm-bound-lem}.
One can prove 
the desired inequality by induction on $n$ using Lemma~\ref{cor:trans}, similarly to the proof of Lemma~\ref{dominate-lem}.\end{proof}
To treat actions of the components of $\cP$ other than $\cP_*$, we will use the following simple relation. 
\begin{lemma}\label{formula}
 For all $n\ge 1$ we have 
\[
\cP^n\circ (I-\pi_0)=\cP_*^n + \sum_{k=0}^{n-1} \cP^k\circ \cP_{10}\circ \cP_{*}^{n-k-1},\]
and
\[
\cP^n\circ\pi_0=\cP_{00}^n+\sum_{k=0}^{n-1}\cP^k\circ  \cP_{01}\circ \cP_{00}^{n-k-1}.
\]
\end{lemma}
\begin{proof}
The claims for $n=1$ are obvious. 
Suppose that the claims  hold for $n=m$. For the first claim, we have
\begin{align*}
    \cP^{m+1}\circ (I-\pi_0)&=\cP\circ 
    \left(\cP_*^m + \sum_{k=0}^{m-1} \cP^k\circ \cP_{10}\circ \cP_{*}^{m-k-1}\right)\\
    &=\cP\circ \cP_*^m + \sum_{k=0}^{m-1} \cP^{k+1}\circ \cP_{10}\circ \cP_{*}^{m-k-1}\\
    &=(\cP_*+\cP_{10})\circ \cP_*^m + \sum_{k=0}^{m-1} \cP^{k+1}\circ \cP_{10}\circ \cP_{*}^{m-k-1}\\
    &=\cP_*^{m+1} + \sum_{k=0}^{m} \cP^k\circ \cP_{10}\circ \cP_{*}^{m-k}.
\end{align*}
where, in the third equality, we have used the fact that the $\hat{H}_0$-component of the image of $\cP_*$ vanishes. For the second claim, we have
\begin{align*}
\cP^{m+1}\circ\pi_0&=\cP^{m}\circ (\cP_{00}+\cP_{01})\\
&=\left(\cP_{00}^m+\sum_{k=0}^{m-1}\cP^k\circ \cP_{01}\circ \cP_{00}^{m-k-1}\right)\circ \cP_{00}
 + \cP^m\circ \cP_{01}\\
&=\cP_{00}^{m+1}+\sum_{k=0}^{m}\cP^k\circ \cP_{01}\circ \cP_{00}^{m-k}.
\end{align*}
Thus we obtain the claims by induction on $m$. 
\end{proof}

For any pair $(\varphi,\psi)$ of functions in $C^\theta([0,1]^3)\cap L^2_0([0,1]^3)$,
we are going to estimate 
 $|\langle \cP^n \varphi,\psi\rangle|$, $n\ge 1$.
Contributions from $\cP_*$
can be treated in the same way as in Section~$2$. 

\begin{lemma}\label{upper-3d}There exists a constant $C_1>0$ such that for any pair $(\varphi,\psi)$ of functions in $ C^\theta([0,1]^3)\cap L^2_0([0,1]^3)$ such that $\varphi\not\equiv0$, we have
\[|\langle \cP_*^n \varphi,\psi\rangle|\leq
 C_1 n^{-3/2}\hat C_\varphi \|\psi\|_{C^\theta}\ 
\text{ for all }n\geq1,\]
and
 \[\| \cP_*^n \varphi\|_{L^1}\leq C_1 n^{-3/2}\hat C_\varphi\ \text{ for all }n\geq1\]
where $\hat{C}_\varphi>0$ is the constant in \eqref{initial-d-2}.\end{lemma}
\begin{proof}
We write
\[\cP_*^n\varphi=\sum_{\ell=1}^\infty \varphi_\ell^{(n)},\ \varphi_\ell^{(n)}\in \hat H_\ell\text{ for }n\geq1\ \text{ and } \
\psi=\sum_{\ell=0}^\infty \psi_\ell,\ \psi_\ell\in \hat H_\ell.\]
An argument similar to that in Section~\ref{pfthm-sec} for $\cP_0$ shows that 
\[\begin{split}
|\langle \cP_*^n \varphi,\psi\rangle|
=& \left|\sum_{\ell=1}^\infty \langle \varphi^{(n)}_\ell, \psi_{\ell}\rangle \right|
\le \sum_{\ell=1}^\infty \| \varphi^{(n)}_\ell\|_{L^1} \|\psi_{\ell}\|_{L^1}\le \sum_{\ell=1}^\infty \| \varphi^{(n)}_\ell\|_{\theta/2}^{(\ell)} \|\psi_{\ell}\|_{\theta/2}^{(\ell)}\\
\leq&2\hat C_\varphi \|\psi\|_{C^\theta}\sum_{\ell=1}^\infty q_\ell^{(n)}2^{-\theta\ell/2}
=2\hat C_\varphi\|\psi\|_{C^\theta}\sum_{\ell=1}^\infty  
\sum_{\substack{\ell'\geq1\\ |\ell'-\ell|\leq n}}
q^{(0)}_{\ell'}p^{(n)}_{\ell'\ell}2^{-\theta\ell/2}\\
\leq& 2\hat C_\varphi\|\psi\|_{C^\theta}\left(\sum_{\max\{\ell,\ell'\}<n^{1/4}} 
\sum_{\substack{\ell'\geq1 \\ |\ell'-\ell|\leq n}}
q^{(0)}_{\ell'}p^{(n)}_{\ell'\ell}2^{-\theta\ell/2}+\sum_{\ell\geq n^{1/4}}2^{-\theta\ell/2}\right)\\
\leq&2\hat C_\varphi\|\psi\|_{C^\theta}\left(\sum_{\max\{\ell,\ell'\}<n^{1/4}} 
\sum_{\substack{\ell'\geq1 \\ |\ell'-\ell|\leq n}}
C_0n^{-3/2}\ell\ell'2^{-\theta\ell/2}+\sum_{\ell\geq n^{1/4}}2^{-\theta\ell/2}\right)\\
\leq& C_1n^{-3/2}\hat C_\varphi \|\psi\|_{C^\theta},
\end{split}\]
provided that we let the constant $C_1$ be sufficiently large
depending only on $\theta$.
For the second inequality we have used \eqref{eq:ellone}.
For the third inequality we have used 
  Lemma~\ref{norm-bound-lem} to bound $\|\psi_\ell\|_{\theta/2}^{(\ell)}$, and Lemma~\ref{dominate-lem-2} to bound $\|\varphi^{(n)}_\ell\|_{\theta/2}^{(\ell)}$. For the fifth inequality we have used Lemma~\ref{lm:num} and
 $q^{(0)}_{\ell'}\leq1$. 
 A proof of the second inequality in the lemma is analogous.
\end{proof}

To treat 
contributions from the remaining terms, we will use the next lemma.
\begin{lemma}\label{eq:im}
Let $u\in L^2_0([0,1]^3)$ and $v\in C^{\theta}([0,1]^3)$. For all $n\geq0$ we have
\[|\langle \cP^n\circ \cP_{10}u,v\rangle|\le 
2^{-\theta n}\|u\|_{L^1}\|v\|_{C^\theta}.  
\]
\end{lemma}
For the proof of this lemma, we first show
\begin{sublemma}\label{lm:ave}
    Suppose $u\in L^2_0([0,1]^3)$ satisfies the condition 
    \[
    \int_{[0,1]} u(x_u,x_c,x_s) {\rm d}x_s=0\quad \text{for Lebesgue a.e. }(x_u,x_c)\in [0,1]^2.
    \]
    Then, for any $v\in C^\theta([0,1]^3)$ we have 
    \[
    |\langle \cP^n u,v\rangle|\le 2^{-\theta n} \|u\|_{L^1}\|v\|_{C^\theta}
    \quad \text{for all }n\ge 0.\]
    
\end{sublemma}
\begin{proof}

Since the $x_s$-direction is contracting by factor $\frac{1}{2}$ under the iteration of $f$, 
for all $(x_u,x_c)\in[0,1]^2$ and all $n\geq0$ we have
\[
\sup_{x_s\in[0,1]} v( f^n(x_u,x_c,x_s))-\inf_{x_s\in[0,1]} v( f^n(x_u,x_c,x_s))\leq 2^{-\theta n}\|v\|_{C^\theta}.
\]
From this estimate and the assumption on $u$,
 we have 
\begin{align*}
&\left|\int_{[0,1]} u(x_u,x_c,x_s)v( f^n(x_u,x_c,x_s)){\rm d}x_s\right|\\&=\left|\int_{[0,1]} u(x_u,x_c,x_s)(v(f^n(x_u,x_c,x_s))-\inf_{x_s\in[0,1]} v( f^n(x_u,x_c,x_s))){\rm d}x_s\right|\\
&\leq 2^{-\theta n}\|v\|_{C^\theta}\int_{[0,1]}|u(x_u,x_c,x_s)|{\rm d}x_s
\end{align*}
for Lebesgue almost every $(x_u,x_c)\in[0,1]^2$.
Since
$\langle \cP^n u,v\rangle=\langle u,v\circ f^n \rangle$, we obtain the conclusion by
integrating both sides of the above inequality with respect to $(x_u,x_c)\in[0,1]^2$.
\end{proof}

\begin{proof}[Proof of Lemma~\ref{eq:im}]
Write $u\in L^2_0([0,1]^3)$ as $u=\sum_{\ell=0}^\infty u_\ell$, $u_\ell\in \hat H_\ell$. Then we have
$\cP_{10}u=\cP_\beta(u_1)$.
From the explicit expression of $\cP_\beta$ on $\hat{H}_1$ in Section~\ref{action-sec}, $\cP_{10}u$
 satisfies the assumption of Sublemma~\ref{lm:ave}.
 Hence the desired inequality holds.
\end{proof}

\subsection{Proof of the Main Theorem for $M=2$.}\label{pf-main}Let $(\varphi,\psi)$ be a pair of functions in $C^\theta([0,1]^3)\cap L^2_0([0,1]^3)$. We may assume $\varphi\not\equiv0$
for otherwise the desired upper bound
\eqref{eq:upperbound} clearly holds.
By Lemma~\ref{formula}, for all $n\geq1$ we have
\[\langle\cP^n\varphi,\psi\rangle=I+I\!I+I\!I\!I+I\!V,\]
where
\[\begin{split}I&=\langle\cP_*^n\varphi,\psi\rangle,\\
I\!I&=\sum_{k=0}^{n-1} \langle\cP^k\circ \cP_{10}\circ \cP_{*}^{n-k-1}\varphi,\psi\rangle,\\
I\!I\!I&=\langle\cP_{00}^n\varphi,\psi\rangle,\\
I\!V&=\sum_{k=0}^{n-1} \langle\cP^k\circ (I-\pi_0)\circ\cP_{01}\circ \cP_{00}^{n-k-1}\varphi,\psi\rangle.\end{split}\]
We estimate these four terms one by one.
 The first inequality in Lemma~\ref{upper-3d} gives
\[|I|\leq
C_1 n^{-3/2}\hat C_\varphi \|\psi\|_{C^\theta}.\]
To the term in the sum in $I\!I$,
 we apply the first inequality in Lemma~\ref{eq:im} and then the second inequality in Lemma~\ref{upper-3d} to deduce that
\begin{align*}
    |I\!I|&\le\sum_{k=0}^{n-1} 2^{-\theta k}\|\cP_{10}\circ \cP_{*}^{n-k-1}\varphi\|_{L^1} \|\psi\|_{C^\theta}\\
    &\leq\sum_{k=0}^{n-1}2^{-\theta k}\|\pi_0\circ \cP\|_{L^1}\|\cP_{*}^{n-k-1}\varphi\|_{L^1} \|\psi\|_{C^\theta}\\
&\le \sum_{k=0}^{n-1}    2^{-\theta k}\|\pi_0\circ \cP\|_{L^1}C_1 \max\{n-k-1,1\}^{-3/2} \hat C_\varphi\|\psi\|_{C^\theta}.
\end{align*}
Combining these two estimates and the elementary estimate
\[
\sum_{k=0}^{n-1}2^{-\theta k}(n-k)^{-3/2}\leq Cn^{-3/2}.
\]
we obtain
\begin{equation}\label{eq:im2}
|I|+|I\!I|\leq C_2n^{-3/2} \hat C_\varphi\|\psi\|_{C^\theta}
\end{equation}
for some constant $C_2>0$ independent of $\varphi$, $\psi$, $n$.

We next estimate the term $I\!I\!I$. It follows from the latter inequality in \eqref{lm:Ppm} for $\ell=0$ and Lemma~\ref{norm-bound-lem} that 
\begin{align*}
\|\cP_{00}^n\varphi\|_{L^1}&\leq\|\cP_{00}^n\varphi_0\|_{L^1}\leq \|\hat{\cP}_{\beta}:\hat{H}_{0, \hat{C}^\theta}\to \hat{H}_{0, \hat{C}^\theta}\|^n\cdot \|\varphi_0\|_{\theta/2}^{(0)}\\
&\leq 2^{-n}\|\varphi_0\|_{\theta/2}^{(\ell)}\leq 2^{-n}\cdot 2\cdot \|\varphi\|_{C^\theta}
\end{align*}
for all $n\geq1$, where $\|\hat{\cP}_{\beta}\colon\hat{H}_{0, \hat{C}^\theta}\to \hat{H}_{0, \hat{C}^\theta}\|$ denotes the operator norm with respect to the norms $\|\cdot\|_{\theta/2}^{(0)}$. 
Hence we obtain
\[|I\!I\!I|\leq\|\cP_{00}^n\varphi\|_{L^1}\|\psi\|_{L^\infty}\leq 2^{-n}C\|\varphi\|_{C^\theta}\|\psi\|_{C^\theta}\]
for some constant $C>0$ independent of $\varphi,\psi,n$.

In order to estimate the last term $I\!V$,
note that $\cP_{01}\circ \cP^{n-k-1}_{00}\varphi\in\hat H_1$ for $0\leq k\leq n-1$.
By \eqref{lm:Ppm} we have
\begin{equation}\label{eq:im300}\|\cP_{01}\circ \cP^{n-k-1}_{00}\varphi\|_{\theta/2}^{(1)}\leq 2^{-n+k}\|\varphi\|^{(0)}_{\theta/2}.\end{equation}
For the terms in the sum in $I\!V$, we apply the estimate \eqref{eq:im2} with $\varphi$ replaced by 
$\cP_{01}\circ \cP^{n-k-1}_{00}\varphi$, and then \eqref{eq:im300}
and Lemma~\ref{norm-bound-lem} to get
\begin{align*}
|I\!V|
&\le C_2\sum_{k=0}^{n-1}  (\max\{k,1\})^{-3/2} \|\cP_{01}\circ \cP^{n-k-1}_{00}\varphi\|_{\theta/2}^{(1)}\|\psi\|_{C^\theta}\\
&\le C_2 \sum_{k=0}^{n-1} (\max\{k,1\})^{-3/2} 2^{-(n-k)}\|\varphi\|_{\theta/2}^{(0)}\|\psi\|_{C^\theta}\\
&\le 2C_2\sum_{k=0}^{n-1}  (\max\{k,1\})^{-3/2} 2^{-(n-k)}\|\varphi\|_{C^\theta}\|\psi\|_{C^\theta}\\
&\le C n^{-3/2} \|\varphi\|_{C^\theta}\|\psi\|_{C^\theta}
\end{align*}
for some constant $C>0$ independent of $\varphi,\psi,n$.
The polynomial upper bound on correlations in \eqref{eq:upperbound}
follows from the estimates above on the terms $I$, $I\!I$, $I\!I\!I$, $I\!V$.

Finally, to prove the lower bound on correlations in \eqref{eq:lowerbound}, we take a pair $(\varphi,\psi)$ of functions in $C^\theta([0,1])\cap L^2_0([0,1])$ that are strictly increasing, and define $\hat{\varphi},\hat{\psi}\in  C^\theta([0,1]^3)\cap L^2_0([0,1]^3)$ by 
$
\hat{\varphi}(x_u,x_c,x_s)=\varphi(x_c)$, 
$\hat{\psi}(x_u,x_c,x_s)=\psi(x_c).
$
Then we have 
$\mathrm{Cor}(\hat{\varphi},\hat{\psi}\circ f^n)=\langle 
\cP_0^n\varphi,\psi\rangle.$
The desired lower bound in \eqref{eq:lowerbound} is a consequence of 
Lemma~\ref{lm:lower}.\qed

\section{Remarks on the case $M>2$}  In this last section
we explain how to modify the argument in the previous sections for the case $M=2$ to prove the Main Theorem for $M>2$. 

\subsection{Upper bound on correlations}
In the case $M=2$, we have used the Haar wavelets and proceeded with explicit computations. 
However, the Haar wavelets is actually not very necessary and we can extend all the arguments in the previous sections to the case $M>2$ at the cost of losing explicitness in computation.
For each integer $\ell\ge 0$ we consider the partition
\[
\xi_\ell=\left\{\left[\frac{k-1}{M^\ell},\frac{k}{M^\ell}\right)\colon k\in\{1,\ldots, M^\ell\}\right\}
\]
 of $[0,1]$ up to Lebesgue null subsets, and 
 let $K_\ell$ denote the set of functions which are constant on each element of $\xi_\ell$. Note that $K_0$ is the one-dimensional space spanned by the constant functions on $[0,1]$. For each $\ell\ge 1$, we define $H_\ell$ to be the $L^2$ orthogonal complement of $K_{\ell-1}$ in $K_\ell$:  
 \[H_\ell=\{u\in K_\ell\colon \langle u,v\rangle=0\text{ for all } v\in K_{\ell-1}\}.\]
  In terms of probability theory, $H_\ell$ is the space of $\xi_{\ell}$-measurable functions whose conditional expectations with respect to the partition $\xi_{\ell-1}$ vanish. 
  If $M=2$, the space $H_\ell$ is the $2^{\ell-1}$ dimensional space spanned by $\chi_{\ell,k}\in K_\ell$ for  $0\le k<2^{\ell-1}$.
We put $H_0=K_0$.
Then we have
\[
K_\ell=\bigoplus_{\ell'=0}^\ell H_{\ell'},
\]
and
\[
L^2_0([0,1])=\overline{\bigoplus_{\ell\geq1}H_\ell}.
\]
Similarly to 
Section~\ref{reduced-sec} and Section~\ref{rPF-sec},  we define an operator 
$\cP_0\colon L^1([0,1])\to L^1([0,1])$ by 
\[
\cP_0=\cP_\alpha+\cP_\beta,
\]where  
\[
\cP_\alpha u(x)=\frac{1}{2}u(Mx-k+1)\quad \text{on $\left[\frac{k-1}{M},\frac{k}{M}\right)$, $k\in\{1,\ldots,M\}$}
\]
and 
\[
\cP_\beta u(x)= \frac{1}{2M}\sum_{k=0}^{M-1}  u\left(\frac{x+k}{M}\right).
\]
Note that $\cP_0(L^2_0([0,1]))\subset L^2_0([0,1])$, $\cP_\alpha(H_\ell)\subset H_{\ell+1}$, $\cP_\beta(H_{\ell+1})\subset H_\ell$ for all $\ell\geq1$
and
$\cP_\beta(H_1)=\{0\},$
and we have the estimates corresponding to 
the inequalities in Lemma~\ref{trans-lem}. Hence we obtain the claim corresponding to 
Theorem~\ref{Th2}, by comparing the transition induced by $\cP_0$ between the components $H_\ell$ with the Markov process of the symmetric simple random walk with an absorbing wall. We then deduce the upper bound in
\eqref{eq:upperbound} from 
Theorem~\ref{Th2} in the same manner as in Section~\ref{sec:pf}. 

\subsection{Lower bound on correlations}
In order to get the lower bound \eqref{eq:lowerbound} in the Main theorem, it is enough to prove Lemma~\ref{lm:lower} in the case $M>2$. For each $\ell\ge 1$, we say a function $u\in H_{\ell}$ is {\it $\xi$-strictly increasing} 
if $u(x)<u(y)$ for two points $x<y$ on $[0,1]$ that belong to an element of $\xi_{\ell-1}$ and belong to distinct elements of $\xi_{\ell}$. The set $\mathcal{C}_{\ell}$ of $\xi$-strictly increasing functions of $H_\ell$ form a convex cone
and 
satisfies $\cP_\alpha(\mathcal{C}_\ell)\subset \mathcal{C}_{\ell+1}$,
$\cP_\beta(\mathcal{C}_{\ell+1})\subset \mathcal{C}_{\ell}$ for all $\ell\geq1$ and 
$\mathcal C_0=\{0\}$.

In order to follow the argument in the case $M=2$, instead of the $L^\infty$ norm we consider the following norm on $H_\ell$:
\[
\|\varphi\|_{*}^{(\ell)}=\max_{I\in \xi_{\ell-1}}\left(\sup_I\varphi-\inf_I\varphi\right).
\]

Let $\varphi\in C^\theta([0,1])\cap L^2_0([0,1])$.
If we write $\cP_0^n \varphi=\sum_{\ell=1}^\infty \varphi_{\ell}^{(n)}$, $\varphi_\ell^{(n)}\in H_\ell$ for $n\ge 1$, then we have  
\begin{equation}\label{transition1-2}
\|\varphi^{(n+1)}_{\ell}\|_{*}^{(\ell)}
\leq\begin{cases}\vspace{1mm}\displaystyle{\frac{1}{2} \|\varphi^{(n)}_{\ell-1}\|_{*}^{(\ell-1)} +
\frac{1}{2} \|\varphi^{(n)}_{\ell+1}\|_{*}^{(\ell+1)}}& \text{ if }\ell\geq2,\\
\displaystyle{\frac{1}{2} \|\varphi^{(n)}_{\ell+1}\|_{*}^{(\ell+1)}}
& \text{ if }\ell=1.
\end{cases}
\end{equation}
If $\varphi$ is strictly increasing on $[0,1]$ in the usual sense, then each component in the decomposition
 $\varphi=\sum_{\ell=1}^\infty \varphi_\ell$, $\varphi_\ell\in H_\ell$
is $\xi$-strictly increasing, and $\varphi_\ell^{(n)}$ is 
$\xi$-strictly increasing too.
Hence, the inequalities in \eqref{transition1-2} are actually equalities. This yields the lower bound in 
Lemma~\ref{lm:lower}.

\subsection*{Acknowledgments}
HT was supported by the JSPS KAKENHI 23K20220. MT was supported by the JSPS KAKENHI 21H00994.
 \bibliographystyle{amsplain}

\end{document}